\documentclass[ejsv2]{imsart}
\usepackage[table]{xcolor}
\usepackage{pgfmath}

\RequirePackage[authoryear, sort&compress]{natbib}
\RequirePackage[colorlinks,citecolor=blue,urlcolor=blue]{hyperref}

 \usepackage{multirow}
 \usepackage{pifont}
 \allowdisplaybreaks

\RequirePackage{graphicx}

 \RequirePackage{enumitem}

\startlocaldefs
\numberwithin{equation}{section}
\theoremstyle{plain}

\newtheorem{theorem}{Theorem}[section]
\newtheorem{lemma}[theorem]{Lemma}
\newtheorem{proposition}[theorem]{Proposition}
\newtheorem{corollary}[theorem]{Corollary}
\theoremstyle{definition}

\newtheorem*{assumption}{Assumption}
\newtheorem{remark}{Remark}
\theoremstyle{remark}

\newcommand{\colorcell}[1]{%
  \pgfmathparse{#1}%
  \ifdim\pgfmathresult pt<0.03pt
    \cellcolor{red!20}{#1}%
  \else
    \ifdim\pgfmathresult pt>0.08pt
      \cellcolor{red!20}{#1}%
    \else
      \ifdim\pgfmathresult pt>0.07pt
        \ifdim\pgfmathresult pt<0.0801pt
          \cellcolor{yellow!30}{#1}%
        \else
          #1%
        \fi
      \else
        \ifdim\pgfmathresult pt>0.029pt
          \ifdim\pgfmathresult pt<0.0701pt
            \cellcolor{green!20}{#1}%
          \else
            #1%
          \fi
        \else
          #1%
        \fi
      \fi
    \fi
  \fi
}

\newcommand{\colorcellPower}[1]{%
  \pgfmathparse{#1}%
  \ifdim\pgfmathresult pt<0.30pt
    \cellcolor{red!30}{#1}%
  \else
    \ifdim\pgfmathresult pt<0.70pt
      \cellcolor{yellow!30}{#1}%
    \else
      \cellcolor{green!30}{#1}%
    \fi
  \fi
}

\newcommand{\R}{\mathbb{R}}
\newcommand{\lb}{\left(}
\newcommand{\rb}{\right)}
\newcommand{\Var}{\operatorname{Var}}
\newcommand{\cov}{\operatorname{cov}}
\newcommand{\E}{\mathbb{E}}
\newcommand{\N}{\mathbb{N}}

\newcommand{\bfSigma}{\mathbf{\Sigma}}

\newcommand{\bfy}{\mathbf{y}}

\newcommand{\cond}{\stackrel{\mathcal{D}}{\to}}
\newcommand{\conp}{\stackrel{\mathbb{P}}{\to}}
\newcommand{\tr}{\operatorname{tr}}

\newcommand{\PR}{\mathbb{P}}

\newcommand{\bfx}{\mathbf{x}}

\newcommand{\bfA}{\mathbf{A}}
\newcommand{\bfB}{\mathbf{B}}
\newcommand{\bfC}{\mathbf{C}}

\newcommand{\bfu}{\mathbf{u}}

\endlocaldefs

\begin{document}
\begin{frontmatter}
\title{Two-Sample Covariance Inference in High-Dimensional Elliptical Models
}
\runtitle{Two-Sample Covariance Inference in High-Dimensional Elliptical Models}

\begin{aug}
\author[A]{\fnms{Nina}~\snm{Dörnemann}\ead[label=e1]{ndoernemann@math.au.dk}\orcid{0009-0003-0486-4383}}
\address[A]{Department of Mathematics,
Aarhus University\printead[presep={,\ }]{e1}}

\runauthor{N. Dörnemann}
\end{aug}

\begin{abstract}
We propose a two-sample test for large-dimensional covariance matrices in generalized elliptical models. The test statistic is based on a U-statistic estimator of the squared Frobenius norm of the difference between the two population covariance matrices. This statistic was originally introduced by \cite{lichen2012} for the independent component model. As a key theoretical contribution, we establish a new central limit theorem for the U-statistics under elliptical data, valid under both the null and alternative hypotheses. 
This result enables asymptotic control of the test level and facilitates a power analysis. To the best of our knowledge, the proposed test is the first such method to be supported by theoretical guarantees for elliptical data. Our approach imposes only mild assumptions on the covariance matrices and does neither require sparsity nor explicit growth conditions on the dimension-to-sample-size ratio.
We illustrate our theoretical findings through applications to both synthetic and real-world data.
\end{abstract}



\end{frontmatter}

\section{Introduction}
The two-sample testing problem for large covariance matrices has attracted considerable attention in recent years, particularly in modern multivariate analysis, random matrix theory, and their applications \citep{ding2024two, zhang2022asymptotic, lichen2012, chenzhangzhong2010, jiang_yang_2013}. Given two independent samples of possibly different sizes with $p$-dimensional covariance matrices $\bfSigma_{n,1}$ and $\bfSigma_{n,2}$, respectively, the testing problem is formulated as
\begin{align} \label{eq_hypothesis}
    &\mathsf{H}_{0,n} : \bfSigma_{n,1} = \bfSigma_{n,2}
    \quad  \textnormal{vs.}  \quad 
    \mathsf{H}_{1,n} : \bfSigma_{n,1} \neq \bfSigma_{n,2} .
\end{align}

\noindent\textbf{Two influential models for high-dimensional data.}
Besides the challenges of high dimensionality, the data analyst rarely has certainty about the underlying data-generating model. 
Modern statistical methodology is therefore often accompanied by theoretical guarantees under general data-generating models that go well beyond the classical normal distribution. 
Among the most prominent classes of data-generating models are elliptical distributions and independent component models. To be precise, we say that a random vector $\bfx$ in $\R^p$ with mean vector $\boldsymbol{\mu}\in \R^p$ and covariance matrix $\bfSigma\in \R^{p\times p}$ follows an \textit{elliptical model (EM)} if it admits a representation of the form 
\begin{align} \label{eq_em}
    \bfx = \xi \bfSigma^{1/2} \bfu + \boldsymbol{\mu}, \tag{EM}
\end{align}
where $\bfu$ is uniformly distributed on the $p$-dimensional unit sphere and $\xi$ is a non-negative scalar random variable independent of $\bfu$ satisfying $\E [ \xi^p]=p.$ In a different modeling approach, $\bfx$
  follows an \textit{independent component model (ICM)}, where $\bfx$ is given by 
  \begin{align} \label{eq_icm}
      \bfx = \bfSigma^{1/2} \bfy + \boldsymbol{\mu} \tag{ICM}
  \end{align}
  for some random vector $\bfy$ in $\R^p$ with independent entries. 
Despite superficial similarities between EMs and ICMs, they different fundamentally in their modeling capacities with important consequences for application areas \citep{wang2023bootstrap}.
Most notably, elliptical models induce dependence across coordinates of the random vector $\xi \bfu$ via the common scaling factor $\xi$, whereas ICMs assume componentwise independence for $\bfy$. More information on advantages and applications of EMs can be found in \cite{gupta2013elliptically, mcneil2015quantitative, elkaroui2009, owen1983elliptical, kariya2014robustness}.

In applications, the data analyst can typically not be sure about the data-generating model. Recent work has begun to address this issue by proposing goodness-of-fit tests for high-dimensional normal data \citep{chen2023normality} and elliptical distributions \citep{wang2024testing}. However, applying such tests prior to the main data analysis raises the classical challenge of pretesting \citep{bancroft1944biases}. Standard remedies like sample splitting are particularly problematic in high-dimensional settings, as they further diminish the effective sample size relative to the dimension.
It is therefore desirable to avoid pretesting altogether and instead turn to procedures that remain valid and powerful across a broad range of data-generating mechanisms.

\noindent \textbf{Our contributions.} 
Our contributions to tackling the aforementioned challenges are threefold: 
\begin{enumerate}[label=(\roman*)]
    \item We revisit the influential test by \cite{lichen2012}, originally developed for ICMs, and establish asymptotic level control and power consistency under EMs. To the best of our knowledge, the proposed test is the first one for the two-sample problem \eqref{eq_hypothesis} in the model \eqref{eq_em} with theoretical guarantees. Moreover, together with the results of \cite{lichen2012}, this test is thus valid for both ICMs and EMs.
    \item The statistical methodology builds on our key theoretical contribution: a central limit theorem for U-statistics estimating the squared Frobenius norm of $\bfSigma_{n,1}-\bfSigma_{n,2}$ under EMs. Notably, the limiting variance generally differs from that in the ICM case (see Figure \ref{fig:histograms}), reflecting the typical challenges of extending results from independent component models to elliptical models.
    \item We conduct a comprehensive numerical study demonstrating that our method is consistently valid and powerful across a wide range of scenarios. In contrast, competing methods developed for ICMs show erratic performance under EMs, being reliable only in some cases but not others. Additionally, we analyze S\&P 500 stock return data, where our method is the only one to detect a difference.
\end{enumerate}

\noindent\textbf{Related work.}
Two-sample testing for high-dimensional covariance matrices has been extensively studied under the assumption that the data follows an ICM. For instance, \cite{zou2021, zheng2012central, zheng_et_al_2015} developed tests based on the eigenvalues of $F$-matrices which requires $p<n_1$ or $p<n_2.$ In this setting, likelihood-based approaches, such as the corrected likelihood ratio tests in \cite{baietal2009, jiang_yang_2013, jiang2015, dettedoernemann2020} for high-dimensional normal distributions have also been proposed, where the distributional assumption was weakened to ICMs in \cite{dornemann2023likelihood}.
The test proposed by \cite{zhang2022asymptotic} combines leading eigenvalues above the phase-transition threshold with linear spectral statistics, and requires equal sample sizes as well as an asymptotically proportional relationship between dimension and sample size for its theoretical guarantees.
The tests developed by \cite{srivastava2010testing} for normally distributed data, \cite{lichen2012} based on U-statistics estimating the squared Frobenius norm of $\bfSigma_{n,1} - \bfSigma_{n,2}$, \cite{cai2013two} for sparse settings
work under more flexible asymptotic regimes. 
\cite{yu2024fisher} Fisher-combined the tests by \cite{lichen2012} and \cite{cai2013two} to enhance power in both sparse and dense settings. 
 \cite{ding2024two} developed a test for the ultra-high dimensional regime where the dimension is much larger than both samples sizes.
In summary, all of these works provide theoretical guarantees under ICM-type models or even the normal distribution, see also Table \ref{table_method}. 
A recent contribution by \cite{li2025tests} considers shape matrices under both EM and ICM and again requires the dimension not to exceed at least one sample size. Note that shape matrices, unlike covariance matrices, are scale-invariant and encode only directional structure.

In recent years, several researchers have begun addressing the challenge of extending results from ICMs to EMs, making it an increasingly active area of research.
A few general references contribute to the understanding of high-dimensional covariance matrices generated by elliptical data. 
For instance, the fluctuations of linear spectral statistics were studied in \cite{hu_et_al_2019,yang_et_al_2021, zhang2022clt}, with a parametric bootstrap procedure provided by \cite{wang2023bootstrap}. Leading eigenvalues under EMs were considered in \cite{jun2022tracy, ding2023extreme}. 
 For most of these results, moving beyond the ICM setting, has historically required considerable time, highlighting the technical challenges involved. 
These challenges are also reflected in the fact that, compared to inference under ICMs, hypothesis testing in EMs remains much less touched.
 Initial progress has primarily focused on one-sample problems in EMs, including sphericity tests \citep{xu2024adjusted} and tests based on linear spectral statistics \citep{hu_et_al_2019, yang_et_al_2021}.
To the best of our knowledge, however, no two-sample test for \eqref{eq_hypothesis} with theoretical guarantees for EMs exists. This work aims to close this gap.

 \noindent\textbf{Overview.} The remainder of this paper is organized as follows. Section \ref{sec_method} introduces the statistical test. Section \ref{sec_theory} presents the model, assumptions, and theoretical guarantees under the null and alternative hypotheses. Section \ref{sec_numerics} investigates finite-sample performance using synthetic data and an application to S\&P 500 stock returns. The proofs of the main theoretical results are gathered in Appendix \ref{sec_proofs}, while additional technical lemmas are placed in Appendix \ref{sec_appendix_b}.

\section{A two-sample test for large-dimensional covariance matrices}

\subsection{Methodology} \label{sec_method}
We consider two independent samples $\bfx_1^{(i)}, \ldots, \bfx_{n_i}^{(i)}$, $i\in \{1,2\},$ with covariance matrices $\bfSigma_{n,1}$ and $\bfSigma_{n,2}$, respectively, and wish to test \eqref{eq_hypothesis}. 

\noindent\emph{Estimating the squared Frobenius norm.}
The proposed test is based on the squared Frobenius norm
\begin{align} \label{eq_decomp}
\| \bfSigma_{n,1} - \bfSigma_{n,2}\|_F^2 := 
    \tr \lb \bfSigma_{n,1} - \bfSigma_{n,2} \rb ^2
    = \tr \lb \bfSigma_{n,1}^2 \rb  + \tr \lb \bfSigma_{n,2}^2 \rb 
    - 2 \tr \lb \bfSigma_{n,1} \bfSigma_{n,2} \rb.
\end{align}
For this purpose, we estimate each term on the right hand side of \eqref{eq_decomp} individually. Define 
\begin{align} 
     U_{n,i} & = \frac{1}{ (n_i)_2}\sum_{1 \leq j , k \leq n_i}^\star \lb \bfx_j^{(i)^\top} \bfx_k^{(i)} \rb^2 
      -\frac{2}{(n_i)_3}\sum_{1 \leq j,k,l \leq n_i  }^{\star} \bfx_j^{(i)^\top}\bfx_k^{(i)}\bfx_k^{(i)^\top}\bfx_{l}^{(i)} \nonumber \\ & \quad 
       +\frac{1}{(n_i)_4}\sum_{1 \leq j, k , l, m \leq n_i }^{\star} \bfx_j^{(i)^\top}\bfx_k^{(i)}\bfx_{l}^{(i)^\top}\bfx_m^{(i)}
     , ~ \quad i \in \{1,2\}, 
     \label{eq_def_U}
\end{align}
 where $\sum^\star$ stands for the summation over pairwise distinct indices and $(n_i)_m:= n_i(n_i-1) \cdot \ldots \cdot (n-m+1)$ for $1 \leq m \leq n_i. $
Then, $U_{n,i}$ is an unbiased estimator for the second spectral moment $\tr \bfSigma_{n,i}^2$ of $\bfSigma_{n,i}$. As an unbiased estimator for $\tr \lb \bfSigma_{n,1} \bfSigma_{n,2} \rb$, we use $V_n,$ where
\begin{align}
    V_n & = \frac{1}{ n_1 n_2} \sum_{j=1}^{ n_1 } \sum_{k=1}^{n_2 } \lb \bfx_i^{(1)^\top} \bfx_j^{(2)} \rb^2
    - \frac{1}{n_1 n_2 (n_1 -1)} \sum_{1 \leq k,l \leq n_1}^\star \sum_{j=1}^{n_2} \bfx_l^{(1)^\top} \bfx_j^{(2)} \bfx_j^{(2)^\top} \bfx_k^{(1)} \nonumber \\
    & \quad  -   \frac{1}{n_2 n_1 (n_2 -1)} \sum_{1 \leq k,l \leq n_2}^\star \sum_{j=1}^{n_1} \bfx_l^{(2)^\top} \bfx_j^{(1)} \bfx_j^{(1)^\top} \bfx_k^{(2)} \nonumber \\ & \quad 
    + \frac{1}{(n_1)_2 (n_2)_2} \sum_{1 \leq i,k \leq n_1}^\star \sum_{1 \leq j,l \leq n_2}^\star 
    \bfx_i^{(1)^\top} \bfx_j^{(2)} \bfx_k^{(1)^\top} \bfx_l^{(2)^\top} 
    .
     \label{eq_def_V}
\end{align}
Then, we propose $T_n$ as an estimator of $\| \bfSigma_{n,1} - \bfSigma_{n,2} \|_F^2,$ where
\begin{align*}
    T_n : =  U_{n,1 } + U_{n,2 } - 2 V_n.
\end{align*}

\noindent\emph{Estimating the variance of $T_n$.}
We will show in Theorem \ref{thm_conv_non_feasible} that the leading order term in the variance of $T_n$ under EMs equals
  \begin{align} 
        \sigma_n^2 & := \sum_{i=1,2} \Bigg\{ 
        \frac{4}{ n_i^2} \tr^2 \bfSigma_{n,i}^2 
         + \frac{8}{n_i} \tr  \left\{ \bfSigma_{n,i} \lb \bfSigma_{n,1} - \bfSigma_{n,2} \rb   \right\} ^2 
         \nonumber \\ & \quad ~ \quad ~ \quad  
        + \frac{4 (\tau_i - 2) }{p n_i}  \tr^2 \lb \bfSigma_{n,i} \lb \bfSigma_{n,1} - \bfSigma_{n,2} \rb \rb 
         \Bigg\} 
        + \frac{8}{n_1 n_2} \tr^2 \lb \bfSigma_{n,1} \bfSigma_{n,2}\rb.  \label{eq_def_sigma_n}
    \end{align}
   Here, $\tau_i$ is a constant independent of $n$ and is related to the variance of the scalar variable in \eqref{eq_em}. 
   It will be formally introduced in \eqref{eq_ass_xi}. 
    Under $\mathsf{H}_{0,n},$ we denote $\bfSigma_{n,0} = \bfSigma_{n,1} = \bfSigma_{n,2}$ and the expression for $\sigma_n^2$ simplifies to 
    \begin{align}
        \sigma_{n,0}^2 & = 4 \lb \frac{1}{n_1} + \frac{1}{n_2} \rb^2 \tr^2 \lb \bfSigma_{n,0}^2 \rb .
    \end{align}
Then, the proposed estimator for $\sigma_{n,0}$ is given by 
\begin{align*}
    \hat\sigma_{n,0} = \frac{2}{n_1} U_{n,1} + \frac{2}{n_2} U_{n,2} .
\end{align*}

\noindent\emph{Test decision.}
For the testing problem \eqref{eq_hypothesis}, we propose to reject $\mathsf{H}_{0,n}$ in favor of $\mathsf{H}_{1,n}$ at a prescribed level $\alpha \in (0,1)$, whenever
\begin{align} \label{eq_test_decision}
    \hat{\mathcal{T}}_{n} := \frac{1}{\hat\sigma_{n,0}}  T_n > z_{1-\alpha} ,
\end{align}
where $z_{1-\alpha}$ denotes the $(1-\alpha)$-quantile of a standard normal distribution. 
    
\subsection{Theoretical results} \label{sec_theory}

\begin{assumption} 
\noindent

\begin{enumerate}[label=(A\arabic*)] 
\item \label{ass_model}
    The samples $\bfx_1^{(1)}, \ldots, \bfx_{n_1}^{(1)}$ and $\bfx_2^{(2)}, \ldots, \bfx_{n_2}^{(2)}$ are independent. 
    For $i=1,2$, the observations $\bfx_1^{(i)}, \ldots, \bfx_{n_i}^{(i)}$ follow a generalized elliptical distribution with mean vector $\boldsymbol{\mu}_{n,i}\in\R^p$ and covariance matrix $\bfSigma_{n,i}\in \R^{p\times p}$, that is,
    \begin{align*}
    \bfx_j^{(i)} = \xi_{i,j} \bfSigma_{n,i}^{1/2} \bfu_{i,j} + \boldsymbol{\mu}_{n,i}, \quad 1 \leq j \leq n_i.
\end{align*}
    Here, $(\bfu_{i,1}, \xi_{i,1}), \ldots, (\bfu_{i,n_i}, \xi_{i,n_i})$ are i.i.d. random vectors in $\R^p \times [0,\infty)$ such that
    $\bfu_{i,1}$ is uniformly distributed on the $p$-dimensional unit sphere, $\xi_{i,1}$ is independent of $\bfu_{i,1}$. Moreover, $\xi_{i,1}$ satisfies the following moment assumptions
    \begin{align}
       \E [ \xi_{i,1}^2]  = p, \quad  
       \Var \lb \frac{\xi_{i,1}^2}{\sqrt{p}} \rb  =  \tau_i + o(1), \quad 
       \E \left| \frac{\xi_{i,1}^2 - p}{\sqrt{p}} \right|^{8} = o(p^2),
       \label{eq_ass_xi}
    \end{align}
    for some fixed $\tau_i>0$ independent of $n$. 
    \item The parameters $p=p(n), n_1 = n_1(n), n_2=n_2(n) \in\N$  are functions of $n\in\N$, and $p, n_1, n_2\to\infty$, as $n\to\infty$. Moreover, $n_1/(n_1 + n_2)$ is asymptotically of order $1.$ \label{ass_n_p}
    \item \label{ass_cov}
    For any $i,j,k,l\in \{1,2\}$, we have that $\tr ( \bfSigma_k \bfSigma_l ) \to \infty$ and
    \begin{align*}
        \tr \lb \bfSigma_{n,i} \bfSigma_{n,j} \bfSigma_{n,k} \bfSigma_{n,l}\rb = o \left\{ \tr \lb \bfSigma_{n,i} \bfSigma_{n,j} \rb \tr \lb \bfSigma_{n,k} \bfSigma_{n,l} \rb \right\} .
    \end{align*}
    \end{enumerate}
\end{assumption}
\begin{remark} {\rm
Assumption \ref{ass_model} appeared in many previous works analyzing high-dimensional elliptical data and describes the way how $\xi_{i,1}^2/\sqrt{p} $ degenerates. Examples of specific distributions satisfying these conditions, as well as further discussion, can be found in \cite{wang2023bootstrap, hu_et_al_2019, wang2024testing, hu2019central}. Assumption \ref{ass_n_p} on the sample sizes $n_1, n_2$ being comparable is conventional for two-sample testing problems \citep{srivastava2010testing, cai2013two, schott2007test, baietal2009}. Notably, it does not impose restriction on the magnitude of the dimension-to-sample size ratio $p/n.$
Assumption \ref{ass_cov} was introduced by \cite{lichen2012}, while \cite{yu2024fisher} impose a stronger version of \ref{ass_cov}.  }
\end{remark}

\begin{figure}[h!]
    \centering
    \includegraphics[width=0.49\linewidth]{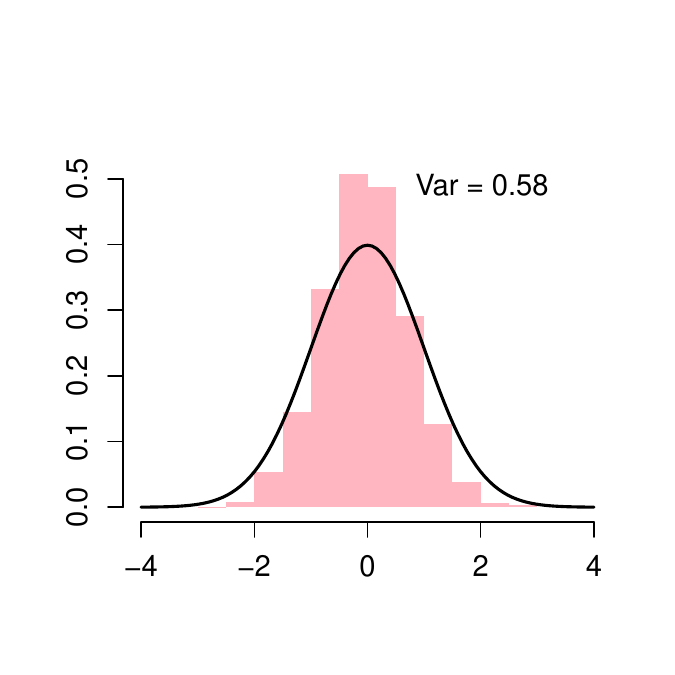}
    \includegraphics[width=0.49\linewidth]{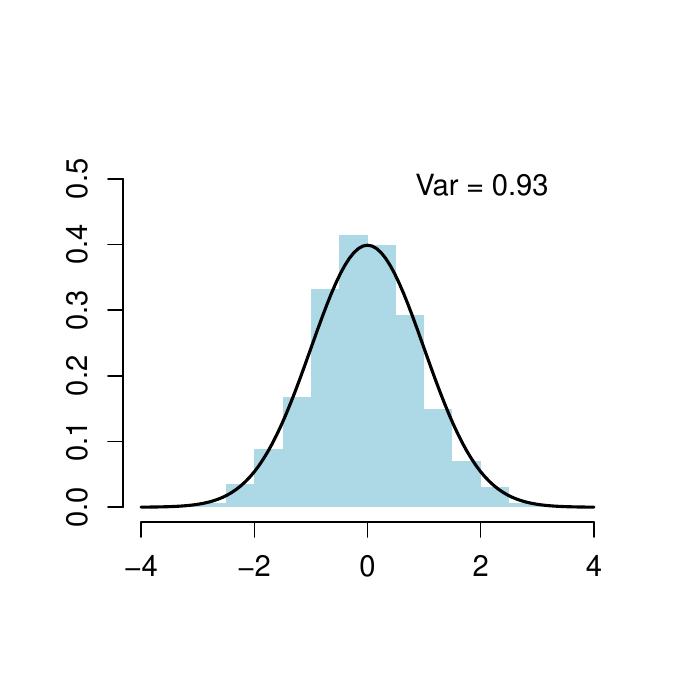}
     \includegraphics[width=0.49\linewidth]{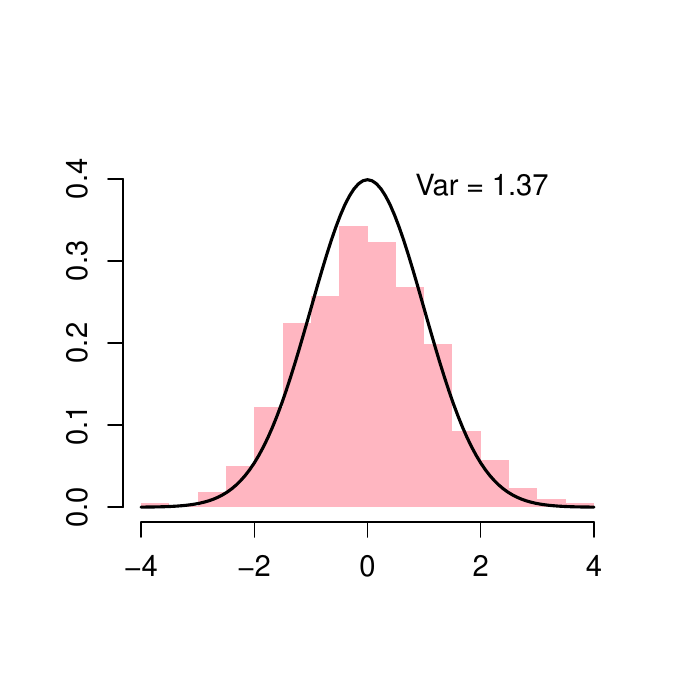}
    \includegraphics[width=0.49\linewidth]{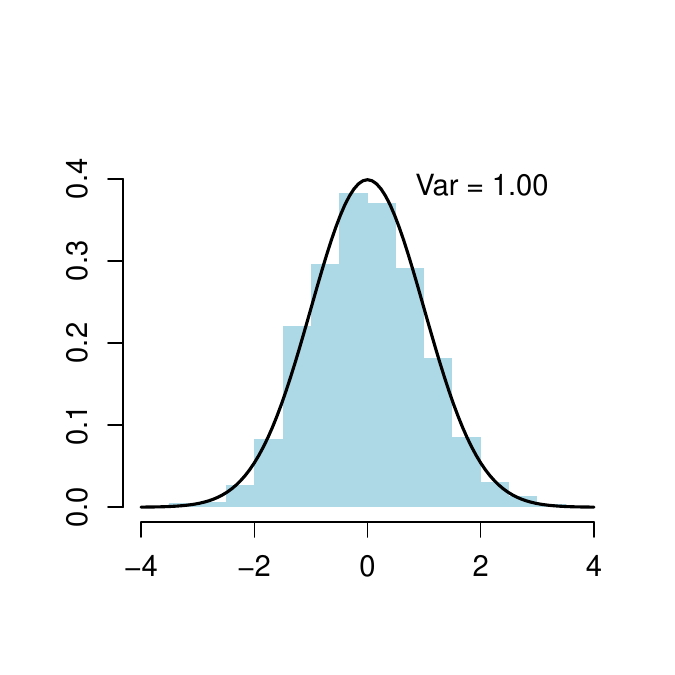}
    \caption{Histograms of    
    $(1/\gamma_n) (T_n - \| \bfSigma_{n,1} - \bfSigma_{n,2}\|_F^2)$ (left column) 
    and $(1/\sigma_n) (T_n - \| \bfSigma_{n,1} - \bfSigma_{n,2}\|_F^2)$ (right column) for $\xi_{1,1}^2, \xi_{1,2}^2 \sim (p+4)\operatorname{Beta}(p/2, 2)$ (upper row), $\xi_{1,1}^2, \xi_{1,2}^2 \sim 1/(p+1)\operatorname{Gamma}( p, 1)^2$ (lower row), $(n_1, n_2,p)=(100,100,200)$, $\bfSigma_{n,1}=\operatorname{diag}(1, \ldots,1)$, $\bfSigma_{n,2}=\operatorname{diag}(5, \ldots,5)$, based on 2000 repititions. }
    \label{fig:histograms}
\end{figure}
\begin{theorem} \label{thm_conv_non_feasible}
Suppose that assumptions \ref{ass_model}- \ref{ass_cov} are satisfied. Then, it holds 
\begin{align*}
    \frac{1}{\sigma_n} \lb T_n - \| \bfSigma_{n,1} - \bfSigma_{n,2} \|_F^2 \rb  \cond Z \sim \mathcal{N}(0,1),
\end{align*}
where $\sigma_n$ is defined in \eqref{eq_def_sigma_n}.
\end{theorem}
The proof of Theorem \ref{thm_conv_non_feasible} is deferred to Section \ref{sec_proof_main}. 
\begin{remark}
{\rm
    \cite{lichen2012} showed a central limit theorem for $T_n$ for an ICM-type model. (In fact, the authors allow for a mild relaxation of the model presented in \eqref{eq_icm} and in particular replace the assumption of independent components by the factorizing of mixed moments up to order 8.) In this case, the asymptotic variance equals
       \begin{align*}
         \gamma_n^2 & := \sum_{i=1,2} \Bigg\{ 
        \frac{4}{ n_i^2} \tr^2 \bfSigma_{n,i}^2 
         + \frac{8}{n_i} \tr  \left\{ \bfSigma_{n,i} \lb \bfSigma_{n,1} - \bfSigma_{n,2} \rb   \right\} ^2 
         \nonumber \\ & \quad ~ \quad ~ \quad  
        + \frac{4 ( \nu_{i,3} -3 )}{n_i} \tr \left\{ \bfSigma_{n,i}^{1/2} \lb \bfSigma_{n,1} - \bfSigma_{n,2} \rb \bfSigma_{n,i}^{1/2} \circ \bfSigma_{n,i}^{1/2} \lb \bfSigma_{n,1} - \bfSigma_{n,2} \rb \bfSigma_{n,i}^{1/2}  \right\}
         \Bigg\} \\ & 
        + \frac{8}{n_1 n_2} \tr^2 \lb \bfSigma_{n,1} \bfSigma_{n,2}\rb. 
    \end{align*}
    Here, $'\circ '$ denotes the Hadamard product, and $\nu_{i,3}$ denotes the fourth moment of the independent component. 
    It is important to emphasize that their result is not valid for elliptical data-generating distributions as captured by assumption \ref{ass_model}.
    In general, their asymptotic variance significantly differs from the variance $\sigma_n^2$ for the EM. This observation is illustrated numerically in Figure \ref{fig:histograms}, where we display histograms of    
    $(1/\gamma_n) (T_n - \| \bfSigma_{n,1} - \bfSigma_{n,2}\|_F^2)$ (left column) 
    and $(1/\sigma_n) (T_n - \| \bfSigma_{n,1} - \bfSigma_{n,2}\|_F^2)$ (right column)
    for two elliptical distributions. The black curve indicates the density of the standard normal distribution. It can be observed, that under  \ref{ass_model}, $1/\gamma_n$ does not standardize $T_n$, with $\gamma_n$ overestimating (upper left) or underestimating (lower left) the variance.

    For the normal case, however, we have $\xi_{1,1}^2, \xi_{2,1}^2\sim \mathcal{X}^2_p$, $\tau_1 = \tau_2 = 2$, and $\nu_{1,3} = \nu_{2,3} = 3$. Thus, the variances $\gamma_n^2$ and $\sigma_n^2$ coincide in this special case. }
\end{remark}
Next, we provide the ratio-consistency of the estimator $\hat\sigma_{n,0}$ for $\sigma_{n,0}$ under the null hypothesis and show that its deterministic approximate under the alternative is given by
\begin{align*}
    \tilde\sigma_{n} & =  \frac{2}{n_1} \tr \lb  \bfSigma_{n,1}^2 \rb + \frac{2}{n_2} \tr \lb  \bfSigma_{n,2}^2 \rb. 
\end{align*}
The proof can be found in Section \ref{sec_proof_power_consistent}.
\begin{proposition} \label{prop_consistent_sigma}
     Suppose that assumptions \ref{ass_model}- \ref{ass_cov} are satisfied. Then, it holds
     \begin{align*}
         \frac{\hat\sigma_{n,0}}{\tilde\sigma_{n}} \conp 1.
     \end{align*}
     In particular, if $\mathsf{H}_{0,n}$ holds true for all large $n$, then  $\hat\sigma_{n,0}$ is a ratio-consistent estimator of $\sigma_{n,0}$, that is,
     \begin{align*}
         \frac{\hat\sigma_{n,0}}{\sigma_{n,0}} \conp 1.
     \end{align*}
\end{proposition}
Recall the definition of $ \hat{\mathcal{T}}_{n}$ in \eqref{eq_test_decision}. Then, Proposition \ref{prop_consistent_sigma} and Theorem \ref{thm_conv_non_feasible} imply that the level of the test proposed in \eqref{eq_test_decision} can be asymptotically controlled. 
\begin{corollary}[Level control] \label{cor_level}
    Suppose that assumptions \ref{ass_model}- \ref{ass_cov} are satisfied. 
  If $\mathsf{H}_{0,n}$ holds true for all large $n,$ we have 
            \begin{align*}
                \lim_{n\to\infty} \PR \lb \hat{\mathcal{T}}_{n} > z_{1-\alpha} \rb  = \alpha.
            \end{align*}
\end{corollary}
\begin{remark}
    The analogue of Corollary \ref{cor_level} was proven in \cite{lichen2012} for ICM-type models. Thus, the test based on $\hat{\mathcal{T}}_{n}$ is valid under both ICMs and EMs. 
\end{remark}
The following theorem specifies when the proposed test attains asymptotically power one. This result will be proven in Section \ref{sec_proof_power_consistent}. 
\begin{theorem}[Power consistency] \label{thm_power}
 Suppose that assumptions \ref{ass_model}- \ref{ass_cov} are satisfied. 
  If $\| \bfSigma_{n,1} - \bfSigma_{n,2} \|_F^2 / \sigma_n \to \infty,$ then
    \begin{align*}
          \lim_{n\to\infty} \PR \lb \hat{\mathcal{T}}_{n} > z_{1-\alpha} \rb 
         = 1.
    \end{align*}
\end{theorem}
Theorem \ref{thm_power} says that the proposed test attains nontrivial power if the signal $\| \bfSigma_{n,1} - \bfSigma_{n,2} \|_F^2$ is large compared to the noise $\sigma_n.$ It is not hard to show that a sufficient condition for the test to be powerful is that 
\begin{align*}
   \tilde\sigma_n =  \frac{1}{n_1} \tr \bfSigma_{n,1}^2 + \frac{1}{n_2} \tr \bfSigma_{n,2}^2 = \mathcal{O} \lb \| \bfSigma_{n,1} - \bfSigma_{n,2} \|_F^2 \rb. 
\end{align*}
Sufficient conditions for the latter are discussed in \cite{lichen2012} and will not be repeated here for the sake of brevity. 

\section{Finite-sample performance} \label{sec_numerics}
In Section \ref{sec_experiments}, we conduct experiments with synthetic data, while we analyze  stock returns from S\&P 500 in Section \ref{sec_stock}.
We compare the proposed test to several state-of-the-art methods for two-sample covariance testing in high dimensions, as summarized in Table \ref{table_method}. 
\begin{table}[h!]
    \centering
    \begin{tabular}{cccc}
    Abbreviation & Article & Data & Asymptotics  \\
    \hline
        LRT & \cite{dornemann2023likelihood} &  ICM & $p/n_1 \to y_1 \in (0,1)$, $p/n_2 \to y_2 \in (0,1)$ \\
         DHW & \cite{ding2024two} & ICM & $p\asymp n_1^{\alpha_1},~ p\asymp n_2^{\alpha_2}$, $\alpha_1,\alpha_2>1$ \\
         SY & \cite{srivastava2010testing} & normal data  & $n_1, n_2 = \mathcal{O}(p^\delta),~ \delta >1/2$ \\
        ZZPZ & \cite{zhang2022asymptotic}  & ICM  & $n_1=n_2, ~ p/n_1 = p/n_2 \to y \in (0,\infty)$ \\
    \end{tabular}
    \caption{Schematic overview about different methods used for comparision in Section \ref{sec_numerics}.}
    \label{table_method}
\end{table}
Notably, none of the existing methods was developed for the elliptical model \eqref{eq_em}, which is reflected in their erratic numerical performance throughout this section: they sometimes perform well, but often do not approximate the nominal level well or fail to detect differences reliably. In contrast, the proposed method -- explicitly designed for elliptical models -- performs reliably across all scenarios under consideration.

\begin{table}[h!]
    \centering
\begin{tabular}{ccccccc|ccccc} \\
    && \multicolumn{5}{c|}{$\mathsf{H}_0$} & \multicolumn{5}{c}{$\mathsf{H}_1$} \\ 
    & & \ref{case1_data} & \ref{case2_data} & \ref{case3_data} & \ref{case4_data} & \ref{case5_data} 
&  \ref{case1_data} & \ref{case2_data} & \ref{case3_data} & \ref{case4_data} & \ref{case5_data} 
    \\  \hline \\
    \multirow{5}{*}{\ref{case1_cov}}
    & LRT & \colorcell{0.05} & \colorcell{0.10} & \colorcell{0.01} & \colorcell{0.36} & \colorcell{0.18}
    & \colorcellPower{0.88} & \colorcellPower{0.93} & \colorcellPower{0.61} & \colorcellPower{0.98} & \colorcellPower{0.96}
\\
        & DHW & \colorcell{0.00} & \colorcell{0.00} & \colorcell{0.00} & \colorcell{0.00} & \colorcell{0.00}

        & \colorcellPower{0.00} & \colorcellPower{0.00} &  \colorcellPower{0.00} & \colorcellPower{0.00} & \colorcellPower{0.00}
        \\
        & SY  & \colorcell{0.00} & \colorcell{0.00} & \colorcell{0.00} & \colorcell{0.00} & \colorcell{0.00}
        & \colorcellPower{0.09} & \colorcellPower{0.06} & \colorcellPower{0.10} & \colorcellPower{0.09} & \colorcellPower{0.09}
        \\
      & ZZPZ   
      & \colorcell{0.03} & \colorcell{0.03} & \colorcell{0.17} & \colorcell{0.04} & \colorcell{0.02}
      & \colorcellPower{0.62} & \colorcellPower{0.45} & \colorcellPower{1.00} &  \colorcellPower{0.33} & \colorcellPower{0.37}
      \\
      & Proposed &   \colorcell{0.06} &  \colorcell{0.06} &  \colorcell{0.04} &  \colorcell{0.05} &  \colorcell{0.04}
      & \colorcellPower{0.75} & \colorcellPower{0.75} & \colorcellPower{0.74} & \colorcellPower{0.75} & \colorcellPower{0.74}
      \\  \\
      \hline \\ 
       \multirow{5}{*}{\ref{case2_cov}}
       & LRT & \colorcell{0.06} & \colorcell{0.11} & \colorcell{0.00} & \colorcell{0.33} & \colorcell{0.22}
       & \colorcellPower{1.00} & \colorcellPower{1.00} & \colorcellPower{1.00} & \colorcellPower{1.00} & \colorcellPower{1.00}
\\ 
        & DHW  & \colorcell{0.00} & \colorcell{0.00} & \colorcell{0.00} & \colorcell{0.00} & \colorcell{0.00}
        &  \colorcellPower{0.99} & \colorcellPower{0.99}  &  \colorcellPower{ 1.00} & \colorcellPower{0.99} & \colorcellPower{0.97}
        \\
        & SY & \colorcell{0.03} & \colorcell{0.02} & \colorcell{0.03} & \colorcell{0.03} & \colorcell{0.02}
& \colorcellPower{1.00} & \colorcellPower{1.00} & \colorcellPower{1.00} & \colorcellPower{1.00} & \colorcellPower{1.00}
        
        \\
      & ZZPZ & \colorcell{0.03} & \colorcell{0.03} & \colorcell{0.05} & \colorcell{0.03} & \colorcell{0.03}
    & \colorcellPower{1.00} & \colorcellPower{1.00} &    \colorcellPower{1.00} &  \colorcellPower{0.99} & \colorcellPower{1.00}
      
      \\
      & Proposed  & \colorcell{0.09} & \colorcell{0.05} & \colorcell{0.04} & \colorcell{0.05} & \colorcell{0.06}
      & \colorcellPower{1.00} & \colorcellPower{1.00} & \colorcellPower{1.00} & \colorcellPower{1.00} & \colorcellPower{1.00}
      \\
     \\ \hline \\ 
       \multirow{5}{*}{\ref{case3_cov}}
       & LRT & \colorcell{0.04} & \colorcell{0.11} & \colorcell{0.01} & \colorcell{0.34} & \colorcell{0.19}
    & \colorcellPower{1.00} & \colorcellPower{1.00} & \colorcellPower{0.98} & \colorcellPower{1.00} & \colorcellPower{1.00}
       \\ 
        & DHW  
        & \colorcell{0.00} & \colorcell{0.00} & \colorcell{0.00} & \colorcell{0.00} & \colorcell{0.00}
& \colorcellPower{0.02} & \colorcellPower{0.03} & \colorcellPower{0.74} & \colorcellPower{0.20} & \colorcellPower{0.68}
        \\
        & SY  
        & \colorcell{0.02} & \colorcell{0.04} & \colorcell{0.03} & \colorcell{0.03} & \colorcell{0.02}
        &  \colorcellPower{0.14} & \colorcellPower{0.11} &  \colorcellPower{0.11} & \colorcellPower{0.13} & \colorcellPower{0.11}
\\
      & ZZPZ 
      & \colorcell{0.04} & \colorcell{0.04} & \colorcell{0.36} & \colorcell{0.02} & \colorcell{0.04}
    & \colorcellPower{0.85} & \colorcellPower{0.68} &  \colorcellPower{1.00} &  \colorcellPower{0.40} & \colorcellPower{0.49}
      \\
      & Proposed 
      & \colorcell{0.05} & \colorcell{0.07} & \colorcell{0.05} & \colorcell{0.06} & \colorcell{0.05}
    & \colorcellPower{0.98} & \colorcellPower{0.98} &  \colorcellPower{0.98} & \colorcellPower{0.97} & \colorcellPower{0.99}
      \\
     
    \end{tabular}

    \caption{Empirical rejection rates of different tests under $\mathsf{H}_0$ and $\mathsf{H}_1 (\delta=0.2)$ for $(n_1, n_2,p)=(300, 300, 100)$.}
    \label{table_n300_p100}
\end{table}

\begin{table}[h!]
    \centering
\begin{tabular}{ccccccc|ccccc} \\
    && \multicolumn{5}{c|}{$\mathsf{H}_0$} & \multicolumn{5}{c}{$\mathsf{H}_1$} \\ 
    & & \ref{case1_data} & \ref{case2_data} & \ref{case3_data} & \ref{case4_data} & \ref{case5_data} 
&  \ref{case1_data} & \ref{case2_data} & \ref{case3_data} & \ref{case4_data} & \ref{case5_data} 
    \\  \hline \\
    \multirow{5}{*}{\ref{case1_cov}}
        & LRT & \colorcell{0.05} & \colorcell{0.09} & \colorcell{0.01} & \colorcell{0.28} & \colorcell{0.17} 
        & \colorcellPower{0.47} & \colorcellPower{0.59} & \colorcellPower{0.24} &  \colorcellPower{0.83} & \colorcellPower{0.70} \\
        & DHW  & \colorcell{0.06} & \colorcell{0.09} & \colorcell{0.06} & \colorcell{0.07} &  \colorcell{0.03}  
       &   \colorcellPower{0.06} & \colorcellPower{0.06} & \colorcellPower{0.07} & \colorcellPower{0.00} & \colorcellPower{0.09}\\
        & SY  &  \colorcell{0.00} & \colorcell{0.00} & \colorcell{0.00} &  \colorcell{0.00} &  \colorcell{0.00} 
        & \colorcellPower{0.00} & \colorcellPower{0.01} & \colorcellPower{0.02} & \colorcellPower{0.02} & \colorcellPower{0.02}\\
      & ZZPZ  &  \colorcell{0.03} & \colorcell{0.03} & \colorcell{0.13} & \colorcell{0.02} &  \colorcell{0.03}  
      & \colorcellPower{0.46} & \colorcellPower{0.34} & \colorcellPower{0.99} & \colorcellPower{0.19} & \colorcellPower{0.25}
      \\
      & Proposed  &  \colorcell{0.08} & \colorcell{0.04} &  \colorcell{0.04}  & \colorcell{0.04} & \colorcell{0.05}
     & \colorcellPower{0.46} & \colorcellPower{0.45} & \colorcellPower{0.48} & \colorcellPower{0.44} & \colorcellPower{0.45} \\  \\
      \hline \\ 
       \multirow{5}{*}{\ref{case2_cov}}
        & LRT & \colorcell{0.05}  & \colorcell{0.08} & \colorcell{0.01} &  \colorcell{0.29} &  \colorcell{0.18} 
       &  \colorcellPower{0.94} & \colorcellPower{0.98} & \colorcellPower{0.87} & \colorcellPower{1.00} & \colorcellPower{1.00}
        \\
        & DHW  & \colorcell{0.11} & \colorcell{0.13} & \colorcell{0.09} &  \colorcell{0.11} &  \colorcell{0.11} 
        & \colorcellPower{0.98} & \colorcellPower{0.96} & \colorcellPower{0.99} & \colorcellPower{0.95} & \colorcellPower{0.97}
        \\
        & SY  & \colorcell{0.03} & \colorcell{0.04} &  \colorcell{0.03} &  \colorcell{0.03} &  \colorcell{0.05} & \colorcellPower{1.00} & \colorcellPower{1.00} & \colorcellPower{1.00} & \colorcellPower{1.00} & \colorcellPower{1.00}
        \\
      & ZZPZ  &   \colorcell{0.03} & \colorcell{0.02} &  \colorcell{0.05} &  \colorcell{0.03} & \colorcell{0.05} 
     & \colorcellPower{0.99} &  \colorcellPower{0.97} & \colorcellPower{1.00} & \colorcellPower{0.91} & \colorcellPower{0.94}
      \\
      & Proposed  & \colorcell{0.04} & \colorcell{0.04} &  \colorcell{0.05} & \colorcell{0.06} &  \colorcell{0.06} 
      &  \colorcellPower{0.96} & \colorcellPower{0.98} & \colorcellPower{0.97} &  \colorcellPower{0.97}  & \colorcellPower{0.97}
      \\
     \\ \hline \\ 
       \multirow{5}{*}{\ref{case3_cov}}
        & LRT & \colorcell{0.05} & \colorcell{0.10} &  \colorcell{0.00} &  \colorcell{0.28} &  \colorcell{0.18} 
      &  \colorcellPower{0.86} & \colorcellPower{0.91} &  \colorcellPower{0.60}  & \colorcellPower{0.98} & \colorcellPower{0.95}
        \\
        & DHW  & \colorcell{0.09} & \colorcell{0.02} & \colorcell{0.08} & \colorcell{0.04} &  \colorcell{0.03} 
        & \colorcellPower{0.11} & \colorcellPower{0.09} &  \colorcellPower{0.22} &  \colorcellPower{0.15} &  \colorcellPower{0.22}
        \\
        & SY  &  \colorcell{0.02} & \colorcell{0.02} & \colorcell{0.02} & \colorcell{0.03} &  \colorcell{0.02} 
        & \colorcellPower{0.09} &  \colorcellPower{0.07} &  \colorcellPower{0.04} & \colorcellPower{0.11} & \colorcellPower{0.08}
      \\
      & ZZPZ  & \colorcell{0.05} & \colorcell{0.04} &  \colorcell{0.29} &  \colorcell{0.04} &  \colorcell{0.04} 
     & \colorcellPower{0.65} & \colorcellPower{0.53} & \colorcellPower{1.00} & \colorcellPower{0.27} & \colorcellPower{0.37}
      \\
      & Proposed  & \colorcell{0.06} & \colorcell{0.07} & \colorcell{0.07} & \colorcell{0.06}  &\colorcell{0.06} 
     & \colorcellPower{0.79} &  \colorcellPower{0.77} & \colorcellPower{0.82} & \colorcellPower{0.81} &  \colorcellPower{0.77}
      \\
     
    \end{tabular}

    \caption{Empirical rejection rates of different tests under $\mathsf{H}_0$ and $\mathsf{H}_1 (\delta=0.2)$ for $(n_1, n_2,p)=(200, 200, 100)$.}
    \label{table_n200_p100}
\end{table}

\begin{table}[h]
    \centering
\begin{tabular}{ccccccc|ccccc} \\
    && \multicolumn{5}{c|}{$\mathsf{H}_0$} & \multicolumn{5}{c}{$\mathsf{H}_1$} \\ 
    & & \ref{case1_data} & \ref{case2_data} & \ref{case3_data} & \ref{case4_data} & \ref{case5_data} 
&  \ref{case1_data} & \ref{case2_data} & \ref{case3_data} & \ref{case4_data} & \ref{case5_data} 
    \\  \hline \\
    \multirow{3}{*}{\ref{case1_cov}}
        & DHW & \colorcell{0.08} & \colorcell{0.01} & \colorcell{0.01} & \colorcell{0.03} & \colorcell{0.01}

        &\colorcellPower{0.33} & \colorcellPower{0.31} & \colorcellPower{ 0.03} & \colorcellPower{0.51} & \colorcellPower{0.50}
        \\
        & SY 
        &  \colorcell{0.00} & \colorcell{0.00} & \colorcell{0.00} & \colorcell{0.00} & \colorcell{0.00}
& \colorcellPower{0.00} & \colorcellPower{0.00} & \colorcellPower{0.01} & \colorcellPower{0.00} & \colorcellPower{0.00}
      \\
      & Proposed  & \colorcell{0.03} & \colorcell{0.06} & \colorcell{0.08} & \colorcell{0.05} & \colorcell{0.05}
& \colorcellPower{0.28} & \colorcellPower{0.32} &  \colorcellPower{0.26} & \colorcellPower{0.27} & \colorcellPower{0.31}
       \\  \\
      \hline \\ 
       \multirow{3}{*}{\ref{case2_cov}}
        & DHW & \colorcell{0.06} & \colorcell{0.01} & \colorcell{0.09} & \colorcell{0.01} & \colorcell{0.02}

        & \colorcellPower{0.94} & \colorcellPower{0.92} & \colorcellPower{1.00} & \colorcellPower{0.87} & \colorcellPower{0.16}
        \\
        & SY & \colorcell{0.02} & \colorcell{0.04} & \colorcell{0.05} & \colorcell{0.06} & \colorcell{0.03}

        & \colorcellPower{0.95}  & \colorcellPower{0.90} & \colorcellPower{0.99} & \colorcellPower{0.80} & \colorcellPower{0.86} \\
      & Proposed  & \colorcell{0.04} & \colorcell{0.05} & \colorcell{0.04} & \colorcell{0.05} & \colorcell{0.05}
      & \colorcellPower{0.73} & \colorcellPower{0.80} & \colorcellPower{0.78} &  \colorcellPower{0.78} &  \colorcellPower{0.75}
      \\
     \\ \hline \\ 
       \multirow{3}{*}{\ref{case3_cov}}
        & DHW & \colorcell{0.00} & \colorcell{0.02} & \colorcell{0.07} & \colorcell{0.01} & \colorcell{0.11}

        & \colorcellPower{0.32} & \colorcellPower{0.90} & \colorcellPower{0.99} & \colorcellPower{0.04} & \colorcellPower{0.26}
        \\
        & SY & \colorcell{0.00} & \colorcell{0.00} & \colorcell{0.01} & \colorcell{0.02} & \colorcell{0.00}

        & \colorcellPower{0.11} &  \colorcellPower{0.11} & \colorcellPower{0.24} & \colorcellPower{0.06}  & \colorcellPower{0.10}\\
      & Proposed & \colorcell{0.05} & \colorcell{0.05} & \colorcell{0.04} & \colorcell{0.07} & \colorcell{0.05}
      & \colorcellPower{0.74} & \colorcellPower{0.71} & \colorcellPower{0.75} & \colorcellPower{0.71} & \colorcellPower{0.71}
      \\
     
    \end{tabular}

    \caption{Empirical rejection rates of different tests under $\mathsf{H}_0$ and $\mathsf{H}_1$ ($\delta=0.3$ for $(n_1, n_2, p)=(50,100, 300)$.}
    \label{table_n1n2}
\end{table}

\subsection{Numerical experiments} \label{sec_experiments}
\noindent \textbf{Design of experiments.}
Following \cite{wang2024testing}, we consider the following distributions for $\xi_{i,1}^2, i=1,2:$
\begin{enumerate}[label=(\roman*)]
    \item \label{case1_data} $ \mathcal{X}_p^2$ ,
    \item \label{case2_data} Beta-Prime$(\frac{p(1+p+3)}{3}, \frac{1+p+6}{3} )$,
    \item \label{case3_data} $(p+4)$Beta$(p/2, 2)$,
    \item \label{case4_data}
    $\operatorname{Gamma} (p/5, 1/5)$,
    \item \label{case5_data} 
    $\frac{1}{p+1} \Gamma (p,1)^2$ .
\end{enumerate}
Note that the above choices for $\xi_i^2$ satisfy $\E[\xi_i^2]=p. $
Under $\mathsf{H}_0$, we consider the following choices for $\bfSigma_{n,1}=\bfSigma_{n,2}$
\begin{enumerate}[label=(\alph*)]
    \item \label{case1_cov} (Two bulk eigenvalues) $\mathbf{Q}_a \operatorname{diag}( \underbrace{2, \ldots, 2}_{p/2}, \underbrace{1, \ldots, 1}_{p/2} )
 \mathbf{Q}_a^\top $,
    \item \label{case2_cov} (Toeplitz matrix) $\bfSigma_{n,1}=\bfSigma_{n,2} $ has entries $\Sigma_{ij} = \rho^{|i-j|}$, $1 \leq i \leq p,$
    \item \label{case3_cov} (Spiked covariance) $\bfSigma_{n,1}=\bfSigma_{n,2} = \mathbf{Q}_b \operatorname{diag}(5,4,3, \underbrace{1, \ldots, 1}_{p-3}) \mathbf{Q}_b^\top $, 
\end{enumerate}
where $\mathbf{Q}_a, \mathbf{Q}_b$ are  randomly generated $p\times p$ orthogonal matrices. For $\mathsf{H}_1,$ we follow \cite{lichen2012, ding2024two} considering banded alternatives. More precisely, $\bfSigma_{n,1}$ is defined according to \ref{case1_cov}-\ref{case3_cov} and we then set $\bfSigma_{n,2} = \bfSigma_{n,1} + \mathbf{B}(\delta),$
where $\mathbf{B}(\delta)$ is a $p\times p$ banded matrix with entries $b_{ij}(\delta)= \delta^2 \mathbf{1}_{\{i=j\}} + \delta \mathbf{1}_{\{|i-j|=1\}} $.

In Tables \ref{table_n300_p100}-\ref{table_n1n2}, we present the empirical rejection rates of the tests under both the null hypothesis $\mathsf{H}_0$
  and the alternative hypothesis $\mathsf{H}_1$ for \ref{case1_cov}-\ref{case3_cov}, \ref{case1_data}-\ref{case5_data} and $(n_1, n_2,p)\in \{(300,300,100), \\ (200,200,100), (50,100,300)\}$. Note that LRT cannot be applied if $p>n_1$ or $p>n_2$ and ZZPZ requires $n_1=n_2$. Thus, these two methods do not appear in Table \ref{table_n1n2}. 
  All reported results are based on $500$ repetitions and we set the level to $\alpha=0.05.$

\noindent\textbf{Color-coding.}
  Under \(\mathsf{H}_0\), cells are color-coded based on how close the empirical rejection rate is to the nominal level \(\alpha\). Green indicates the rate falls within the 95\% confidence interval, yellow means it falls outside the 95\% but within the 99\% confidence interval, and red means it lies outside the 99.9\% confidence interval. These intervals are based on testing 
$
H_0: q = 0.05$ versus $H_1: q \neq 0.05
$
using 500 observations from a Bernoulli distribution with success probability \(q\), and are rounded to two digits.
To illustrate the power of the tests under the alternative hypothesis $\mathsf{H}_1$, we again use a three-color scheme in the tables. Cells are colored green for high power ($\geq 70\%$), yellow for moderate power ($30\%-69\%$), and red for low power ($<30\%$). 

\noindent\textbf{Numerical results under $\mathsf{H}_0$.} 
In the vast majority of cases under consideration, that is, 41 out of 45, the proposed method closely approximates the nominal level even in finite samples (green cells). In contrast, the performance of the competing methods is more mixed.
Case \ref{case1_data} refers to normally distributed data, which can be modelled both through ICMs and EMs. Thus, all methods are expected to perform well. Indeed, most approaches achieve reasonable level control here, though often slightly conservatively. 
For the other elliptical distributions \ref{case2_data}-\ref{case5_data}, however, the LRT, often fails to maintain the nominal level. The DHW, SY and ZZPZ methods generally perform acceptably, though sometimes conservative, but in several instances exhibit notable exceedances of the nominal level. For example, see case \ref{case3_data}\ref{case3_cov} for ZZPZ in Tables \ref{table_n300_p100} and \ref{table_n200_p100}, and case \ref{case2_cov} in Table \ref{table_n200_p100} for DHW.

\noindent\textbf{Numerical results under $\mathsf{H}_1$.} Under the alternative for scenario \ref{case2_cov}, all methods are very powerful, expect DHW for \ref{case5_data} in Table \ref{table_n1n2}. In Table \ref{table_n300_p100}, cases \ref{case1_cov} and \ref{case3_cov}, DHW and SY have in 8 out of 10 cases no significant power, while LRT and the proposed method are generally powerful. Note that in many of these cases, the high power of the LRT comes at the expense of poor level control under $\mathsf{H}_0$.
If the dimension is closer the sample sizes (Table \ref{table_n200_p100}) or even exceeds them  (Table \ref{table_n1n2}), all methods have low or moderate power for detecting the alternative in \ref{case1_cov} in most cases, with the proposed method being mostly in the moderate regime. 
Moreover, still in Tables \ref{table_n200_p100} and \ref{table_n1n2}, the proposed method is the only one to detect alternative in case \ref{case3_cov} reliably with high power.
\begin{figure}[h!]
    \centering
    \includegraphics[width=0.49\linewidth]{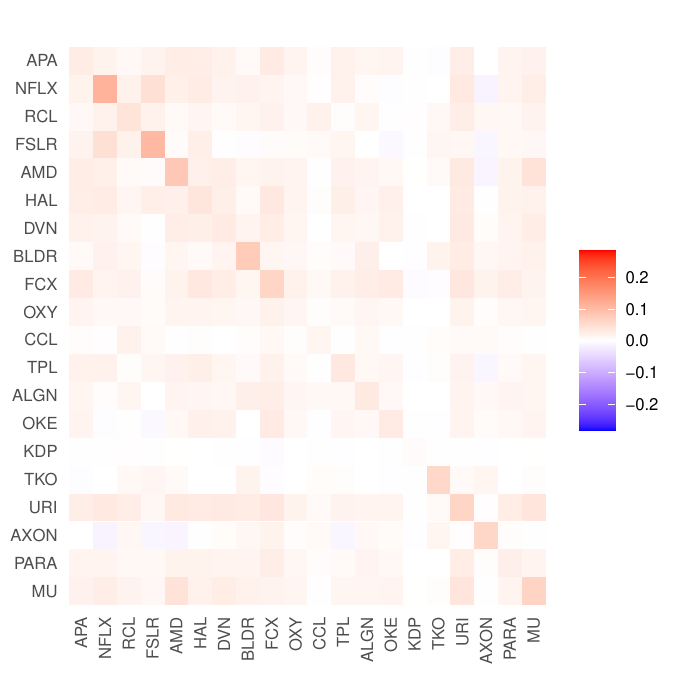}
    \includegraphics[width=0.49\linewidth]{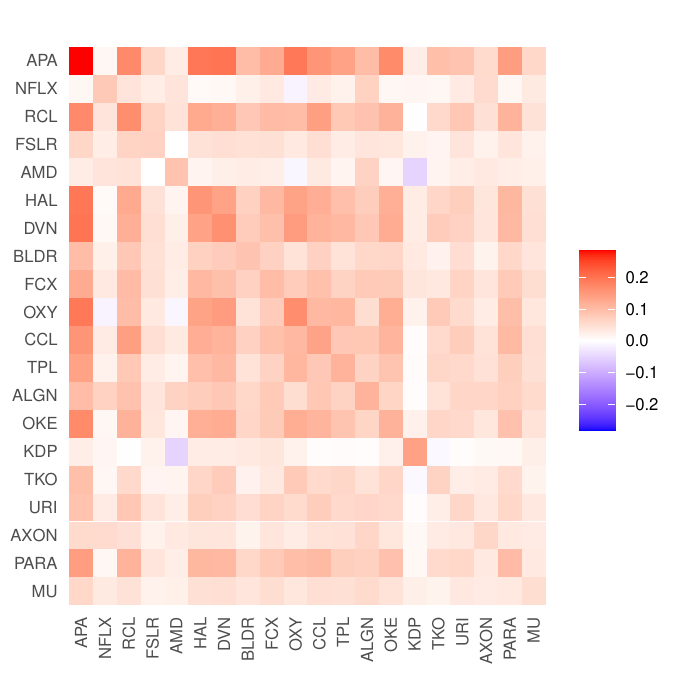}
    \caption{Heatmaps of the sample covariance matrices for the periods 2010 Q1–2017 Q2 (left figure) and 2017 Q3–2022 Q4 (right figure) of the S\&P 500 returns studied in Section \ref{sec_stock} We display only the top 20 stocks with the highest marginal variances. The full covariance matrices are of dimension $434\times 434.$
    }
    \label{fig:heatmaps}
\end{figure} 

\noindent\textbf{Summary.} The numerical study illustrates that tests rooted in the ICM cannot be expected to generalize to EM without suitable adaptations. Competing methods exhibit erratic behavior: they perform well in some scenarios, moderately or poorly in others. In contrast, the proposed method demonstrates reliable performance across a diverse range of covariance structures and data-generating distributions. This inconsistency among existing approaches highlights the need for methods with theoretical guarantees under elliptical models.

\subsection{Application to S\&P-500 stock data} \label{sec_stock}
Elliptical models are frequently used in the analysis of financial data \citep{gupta2013elliptically, mcneil2015quantitative, owen1983elliptical}. One key reason is their ability to capture joint movements in asset returns: in the elliptical model \eqref{eq_em}, the scalar random variable $\xi$ driving the dispersion is shared across all coordinates, which makes it well-suited for modeling events where many assets experience simultaneous increases or drops in magnitude.
In this section, we analyze stock return data for S\&P 500 downloaded from Yahoo Finance, covering the period from January 2010 to December 2022. From these, we compute log quarterly returns . To ensure comparability over time, we retain only those stocks with complete return histories over the entire period, yielding a dataset with $p=434$ stocks.
We then apply the test proposed in this work to assess a change in the covariance structure between 2010 Q1 - 2017 Q2 and 2017 Q3 - 2022 Q4, meaning $n_1=30, n_2=22$ (referred to as Case 1). To visualize the difference between the two periods, we show heatmaps of the sample covariance matrices in Figure \ref{fig:heatmaps}. For clarity, we restrict the display to the $20$ stocks with the highest marginal variances; note that the full covariance matrix is of size 
$434 \times 434$, so the overall change is more subtle.
We also study the periods between 2013 Q1 - 2016 Q1 and 2016 Q2 - 2022 Q4 with smaller samples sizes $n_1=13, n_2=27$ (referred to as Case 2). 

In the two cases, our test rejects the null of equal covariance matrices at both the 5\% and 10\% significance levels. In contrast, the tests of SY and DHW do not reject at either level, see Table \ref{table_stock} for on overview. Note that the ZZPZ-test cannot be applied since it requires $n_1=n_2$, while the LRT-test requires $p<n_1,n_2$.

\begin{table}[h!]
    \centering
    \begin{tabular}{ccc|cc}
         &  \multicolumn{2}{c|}{$\alpha=0.05$} & \multicolumn{2}{c}{$\alpha=0.1$} \\
         & Case 1 & Case 2 & Case 1 & Case 2 \\ 
        DHW &  \ding{55}  & \ding{55} & \ding{55} & \ding{55} \\
        SY   & \ding{55}  & \ding{55} & \ding{55} & \ding{55}\\
        Proposed   & \ding{51}  & \ding{51} & \ding{51}  & \ding{51}\\
    \end{tabular}
    \caption{Test decisions for \eqref{eq_hypothesis} for the S\&P 500 stock returns studied in Section \ref{sec_stock}. "\ding{51}" stands for a rejection of $\mathsf{H}_0$ by the corresponding method, while "\ding{55}" indicates a non-rejection. }
    \label{table_stock}
\end{table}


\begin{appendix}
 \section{Proofs of main results} \label{sec_proofs}

In our proofs, we use a notation to compare the size of two sequences of positive numbers $(a_n)_{n \in \mathbb{N}}$ and $(b_n)_{n \in \mathbb{N}}$. We write $a_n \lesssim b_n$, if there exists a non-random constant $C>0$, independent of $n$ such that $a_n \leq C\, b_n$ for all $n\in \N$.
Moreover, we suppress the dependence on $n$ of $\bfSigma_{n,i}$ in our notation, and set  $\bfSigma_{n,i}=:\bfSigma_i$ for $i\in \{1,2\}$. 
 
\subsection{Proof of Theorem \ref{thm_conv_non_feasible}} \label{sec_proof_main}

\noindent\textbf{Reduction of $T_n$.} 
Note that the statistic $T_n$ is invariant under location shift (see also \cite{lichen2012}), and thus, we may and will assume w.l.o.g. that $\boldsymbol{\mu}_{n,1} = \boldsymbol{\mu}_{n,2} = \mathbf{0}.$ In this case, the leading order terms in the variance of $T_n$ are contributed by the first terms in $U_{n,i}$ and $V_n$, respectively. That is, it suffices to analyze $\tilde T_n = U_{n,1,1} + U_{n,2,1} - 2 V_{1,n},$ where 
\begin{align}
    U_{n,i,1} & =  \frac{1}{ (n_i)_2}\sum_{1 \leq j , k \leq n_i}^\star \lb \bfx_j^{(i)^\top} \bfx_k^{(i)} \rb^2 
    \quad i = 1,2, \label{def_u_ni1} \\ 
    V_{1,n} & = \frac{1}{ n_1 n_2} \sum_{j=1}^{ n_1 } \sum_{k=1}^{n_2 } \lb \bfx_i^{(1)^\top} \bfx_j^{(2)} \rb^2
    .
    \label{def_v_n1}
\end{align}
More formally, the two main steps in the proof of Theorem \ref{thm_conv_non_feasible} lie in the verification of 
\begin{align}
    \label{eq_clt_tilde_T_n}
       \frac{1}{\sigma_n} \lb \tilde T_n - \| \bfSigma_1 - \bfSigma_2 \|_F^2 \rb  \cond Z \sim \mathcal{N}(0,1)
\end{align}
and 
\begin{align} \label{eq_remainder_negl}
    T_n - \tilde T_n = o_{\PR}(\sigma_n). 
\end{align}
The latter fact \eqref{eq_remainder_negl} will be proven in Section \ref{sec_proof_remainder}, while we proceed in the following with a preparatory step for the proof of \eqref{eq_clt_tilde_T_n} via the martingale CLT. 

\noindent\textbf{Preparations for the proof of \eqref{eq_clt_tilde_T_n}.}
Recall that we assume w.l.o.g. that $\boldsymbol{\mu}_{n,i}=\mathbf{0}$.
To facilitate the application of the martingale CLT, we proceed with some preparations. By the means of a martingale decomposition, we write $\tilde T_n$ as a sum of martingale differences. 
For this purpose, we define the triangular array of random variables
\begin{align*}
    \bfx_{j,n} = \bfx_j  = \begin{cases}
        \bfx_{j}^{(1)} & : \textnormal{ if } 1 \leq j \leq n_1, \\
        \bfx_{j - n_1 }^{(2)} & : \textnormal{ if } n_1 +1 \leq j \leq n_1 + n_2,
    \end{cases}
    \quad 1 \leq j \leq n_1 + n_2. 
\end{align*}
Moreover, let $\E_k$ denote the conditional expectation with respect to the $\sigma-$field $\mathcal{F}_k$ generated by $\bfx_1, \ldots, \bfx_k$, $1 \leq k \leq n_1 + n_2.$  More precisely, we have
\begin{align*}
    \tilde T_n =  T_{n,1} + T_{n,2} + T_{n,3},
\end{align*}
where 
\begin{align*}
      T_{n,1} & := \sum_{k=1}^{ n_1 + n_2}  T_{n,k,1}(t) 
      := \sum_{k=1}^{n_1 + n_2} ( \E_k - \E_{k-1} ) \left[  U_{n,1,1}(t) \right] ,\\
     T_{n,2} & := \sum_{k=1}^{n_1 + n_2}  T_{n,k,2}(t)  
     := \sum_{k=1}^{n_1 + n_2} ( \E_k - \E_{k-1} ) \left[  U_{n,2,1}(t)   \right] , \\
    T_{n,3} & := \sum_{k=1}^{n_1 + n_2}  T_{n,k,3} (t) 
    := - 2\sum_{k=1}^{n_1 + n_2} ( \E_k - \E_{k-1} ) \left[  V_{1,n}(t)  \right],
\end{align*}
where $ U_{n,i}, V_n $ are defined in \eqref{eq_def_U} and \eqref{eq_def_V}, respectively. 
 In the following, we will find more handy representations for $T_{n,k,i}$, $i\in \{1,2,3\}$,  by using basic properties of the conditional expectation  
For the first term, we have for $k\leq n_1$
\begin{align*}
    T_{n,k,1} & = \frac{1}{n_1 ( n_1 - 1 )} \sum_{\substack{i,j=1 \\ i \neq j }}^{ n_1 } ( \E_k - \E_{k-1} ) \lb \bfx_i^\top \bfx_j \rb^2 \\ 
    & =   \frac{2}{n_1 ( n_1 - 1 )} \sum_{\substack{i=1 \\ i \neq k }}^{ n_1 } ( \E_k - \E_{k-1} ) \lb \bfx_i^\top \bfx_k \rb^2
    \\ 
    & = \frac{2}{n_1 ( n_1 - 1 )} \sum_{\substack{i=1 }}^{k-1} ( \E_k - \E_{k-1} ) \lb \bfx_i^\top \bfx_k \rb^2 
    +  \frac{2}{n_1 ( n_1 - 1 )} \sum_{\substack{i=k+1 }}^{n_1} ( \E_k - \E_{k-1} ) \lb \bfx_i^\top \bfx_k \rb^2 
    \\ 
    & = \frac{2}{n_1 ( n_1 -1) } \Big\{ \sum_{i=1}^{k-1} \lb \lb \bfx_i^\top \bfx_k \rb^2 - \bfx_i^\top \bfSigma_1 \bfx_i \rb 
    + ( n_1 - k) \lb \bfx_k^\top \bfSigma_1 \bfx_k - \tr \bfSigma_1^2 \rb 
    \Big\} \\ 
    & = \frac{2}{n_1 ( n_1 -1) }  \sum_{i=1}^{k-1} \left\{  \bfx_k^\top \lb \bfx_i \bfx_i^\top - \bfSigma_1 \rb \bfx_k -  \tr \lb  \bfSigma_1 \lb \bfx_i \bfx_i^\top  - \bfSigma_1 \rb \rb  \right\}
    \\ & ~ +  \frac{2 ( k -1) }{n_1 ( n_1 -1) }
    \lb \bfx_k^\top \bfSigma_1 \bfx_k - \tr \bfSigma_1^2 \rb 
     +  \frac{2 ( n_1 - k ) }{n_1 ( n_1 -1) }
    \lb \bfx_k^\top \bfSigma_1 \bfx_k - \tr \bfSigma_1^2 \rb 
    \\ 
    & = \frac{2}{n_1 ( n_1 -1) }  \sum_{i=1}^{k-1} \left\{  \bfx_k^\top \lb \bfx_i \bfx_i^\top - \bfSigma_1 \rb \bfx_k -  \tr \lb  \bfSigma_1 \lb \bfx_i \bfx_i^\top  - \bfSigma_1 \rb \rb  \right\}
    \\ ~& +  \frac{2  }{n_1  }
    \lb \bfx_k^\top \bfSigma_1 \bfx_k - \tr \bfSigma_1^2 \rb.
\end{align*}
For $k \geq n_1 +1,$ it holds $T_{n,k,1} = 0.$
Similarly to $T_{n,1}$, we get for the second term $T_{n,k,2}=0$ for $k \leq n_1$ and for $k\geq n_1 +1$, we have 
\begin{align*}
    T_{n,k,2} & = \frac{2}{ n_2 ( n_2 -1) }  \sum_{i= n_1 +1}^{k-1} \left\{  \bfx_k^\top \lb \bfx_i \bfx_i^\top - \bfSigma_2 \rb \bfx_k -  \tr \bfSigma_2 \lb \bfx_i \bfx_i^\top  - \bfSigma_2 \rb \right\}
    \\ ~& +  \frac{2  }{n_2  }
    \lb \bfx_k^\top \bfSigma_2 \bfx_k - \tr \bfSigma_2^2 \rb .
\end{align*}
We now turn to the computation of the third term. Let us first assume that $k \leq n_1.$ In this case, we get
\begin{align*}
    - \frac{1}{2}  T_{n,k,3} & = \frac{1}{n_1 n_2 } \sum_{i=1}^{ n_1} \sum_{l= n_1+1}^{n_1+ n_2} ( \E_k - \E_{k-1} ) \tr \bfx_i \bfx_i^\top \bfx_l \bfx_l^\top \\
    & = \frac{1}{n_1 n_2 } \sum_{i=1}^{k-1} \sum_{l=n_2}^{n_1 + n_2} ( \E_k - \E_{k-1} ) \tr \bfx_i \bfx_i^\top \bfx_l \bfx_l^\top
    \\ & \quad + \frac{1}{n_1 n_2 } \sum_{i=k+1 }^{ n_1} \sum_{l=n_1+1}^{n_1 + n_2} ( \E_k - \E_{k-1} ) \tr \bfx_i \bfx_i^\top \bfx_l \bfx_l^\top \\ 
    & \quad + \frac{1}{n_1 n_2 }  \sum_{l=n_1+1}^{n_1 + n_2} ( \E_k - \E_{k-1} ) \tr \bfx_k \bfx_k^\top \bfx_l \bfx_l^\top \\ 
    & = \frac{1}{n_1 n_2 }  \sum_{l=n_1+1}^{n_1 + n_2} ( \E_k - \E_{k-1} ) \tr \bfx_k \bfx_k^\top \bfx_l \bfx_l^\top \\
    & = \frac{ 1}{n_1 } \lb \bfx_k^\top \bfSigma_2 \bfx_k - \tr \bfSigma_1 \bfSigma_2 \rb.
\end{align*}
Similar calculations show in the case $k \geq n_1+1$ that
\begin{align*}
   - \frac{1}{2} T_{n,k,3} &= \frac{1}{n_1 n_2 } \sum_{i=1}^{ n_1} \sum_{l=n_1+1}^{n_1 + n_2} ( \E_k - \E_{k-1} ) \tr \bfx_i \bfx_i^\top \bfx_l \bfx_l^\top \\
   & = \frac{1}{n_1 n_2 } \sum_{i=1}^{ n_1}  ( \E_k - \E_{k-1} ) \tr \bfx_i \bfx_i^\top \bfx_k \bfx_k^\top \\
   & = \frac{1}{n_1 n_2 }  \left[   \bfx_k \lb \sum_{i=1}^{ n_1} \bfx_i \bfx_i^\top \rb  \bfx_k  - 
   \tr \bfSigma_2 \lb \sum_{i=1}^{n_1} \bfx_i \bfx_i^\top \rb  
   \right]  \\
   &=  \frac{1}{n_1 n_2 }  \left[   \bfx_k \lb \sum_{i=1}^{n_1} \bfx_i \bfx_i^\top - \bfSigma_1 \rb  \bfx_k  - 
   \tr \bfSigma_2 \lb \sum_{i=1}^{n_1} \bfx_i \bfx_i^\top - \bfSigma_1 \rb  
   \right] \\  & \quad 
   +  \frac{1}{  n_2 } \left[ \bfx_k^\top \bfSigma_1 \bfx_k - \tr \bfSigma_1 \bfSigma_2
   \right].
\end{align*}
Now define 
\begin{align}
    \bfA_{l':l} = \sum_{i=l'}^{l} \lb  \bfx_i \bfx_i^\top - \bfSigma_{\kappa(l)} \rb , \quad 1 \leq l' \leq l \leq n_1 + n_2 , 
    \label{eq_def_A_1k}
\end{align}
where $\kappa(l)=1$ for $1 \leq l \leq n_1$ and $\kappa(l)=2$ for $n_1+1 \leq l \leq {n_1 + n_2}.$
Then, in summary, we have
\begin{align}
    T_{n,k} & := T_{n,k,1} + T_{n,k,2} + T_{n,k,3},
    \label{eq_t_nk_formula}
\end{align}
where for $1 \leq k \leq n_1$ and $n_1 +1 \leq k' \leq n_1 + n_2$,
\begin{align}
    T_{n,k}
    & =  \frac{2}{n_1 ( n_1 -1) }   \left\{  \bfx_k^\top \bfA_{1:(k-1)} \bfx_k -  \tr \bfSigma_1 \bfA_{1:(k-1)} \right\}
    + \frac{2  }{n_1  }
    \lb \bfx_k^\top \bfSigma_1 \bfx_k - \tr \bfSigma_1^2 \rb
   \nonumber \\ & \quad  - \frac{2}{n_1} \lb \bfx_k^\top \bfSigma_2 \bfx_k - \tr \bfSigma_1 \bfSigma_2 \rb
   \nonumber \\ & 
   =  \frac{2}{n_1 }   \Bigg\{  \bfx_k^\top \lb \frac{1}{n_1 - 1} \bfA_{1:(k-1)} + \bfSigma_1 - \bfSigma_{2} \rb \bfx_k \nonumber \\ & 
   -  \tr \bfSigma_1 \lb \frac{1}{ n_1 - 1} \bfA_{1:(k-1)} + \bfSigma_1 - \bfSigma_{2} \rb \Bigg\}
    , \label{eq_t_nk} \\
    T_{n,k'} & = \frac{2}{ n_2 ( n_2 -1) }  \left\{  \bfx_{k'}^\top \bfA_{(n_1+1):(k'-1)} \bfx_{k'} -  \tr \bfSigma_2 \bfA_{(n_1+1) : (k'-1) } \right\} \nonumber \\
    & \quad 
   +  \frac{2  }{n_2  }
    \lb \bfx_{k'}^\top \bfSigma_2 \bfx_{k'} - \tr \bfSigma_2^2 \rb -  \frac{2}{n_1 n_2 }  \left[   \bfx_{k'} \bfA_{1:n_1} \bfx_{k'}  - 
   \tr \bfSigma_2 \bfA_{1:n_1} 
   \right] \nonumber  \\ 
    & \quad  -  \frac{2}{ n_2 } \left[ \bfx_{k'}^\top \bfSigma_1 \bfx_{k'} - \tr \bfSigma_1 \bfSigma_2
   \right] \nonumber \\ 
   & = \frac{2}{  n_2  }  \Bigg\{  
   \bfx_{k'}^\top  \lb \frac{1}{ n_2 -1}  \bfA_{(n_1+1):(k'-1)} - \frac{1}{n_1} \bfA_{1:n_1} + \bfSigma_2 - \bfSigma_1 \rb  \bfx_{k'} \nonumber \\ 
   & \quad -  \tr \bfSigma_2 \lb \frac{1}{ n_2 -1}  \bfA_{(n_1+1):(k'-1)} - \frac{1}{n_1} \bfA_{1:n_1} + \bfSigma_2 - \bfSigma_1 \rb  \Bigg\}.  \label{eq_t_nk'}
\end{align}
We now turn to the computation of the covariance kernel. 
By the martingale CLT and \eqref{eq_t_nk_formula}, we need to compute 
\begin{align*}
    \sum_{k=1}^{n_1 + n_2} \sigma_{n,k}  := \sum_{k=1}^{n_1 + n_2} \E_{k-1} \left[  T_{n,k}^2 \right] . 
\end{align*}
We will show the following auxiliary results. Recall the definition of $\sigma_n^2$ in \eqref{eq_def_sigma_n}. 
\begin{lemma} \label{lem_cov_mean}
Under the assumptions of Theorem \ref{thm_conv_non_feasible}, it holds
    \begin{align*}
         & \E \left[ \sum_{k=1}^{n_1 + n_2}  \sigma_{n,k}^2 \right] 
          =  \sigma_n^2
    + o \lb \sigma_n^2 \rb .
    \end{align*}
\end{lemma}

\begin{lemma}
    \label{lem_cov_var}
   Under the assumptions of Theorem \ref{thm_conv_non_feasible}, it holds
    \begin{align*}
        \E \left[  \sum_{k=1}^{n_1 + n_2}  \sigma_{n,k}^2 \right] - \sum_{k=1}^{n_1 + n_2}  \sigma_{n,k}^2 = o_{\PR}(\sigma_n^2). 
    \end{align*}
\end{lemma}
From Lemma \ref{lem_cov_mean} and Lemma \ref{lem_cov_var}, it follows that 
\begin{align}
    \label{eq_conv_var}
   \frac{ \sum_{k=1}^{n_1 + n_2}  \sigma_{n,k}^2  }{ \sigma_n^2} \conp  1.
\end{align}
 The following lemma says that $\sum_{k=1}^{n_1 + n_2} T_{n,k} / \sigma_n $ satisfies the Lindeberg-type condition. 
\begin{lemma} 
\label{lem_lindeberg}
Under the assumptions of Theorem \ref{thm_conv_non_feasible}, it holds
    \begin{align*}
        \frac{   \sum_{k=1}^{n_1 + n_2}  \E [ T_{n,k}^4  ] }{ \sigma_n^4} = o(1) .
    \end{align*}
\end{lemma}

From \eqref{eq_conv_var} and Lemma \ref{lem_lindeberg}, we conclude \eqref{eq_clt_tilde_T_n} by an application of the martingale CLT \citep[Theorem 18.1]{billingsley1999}. Thus, the next task lies in verifying the lemmas stated in this section. 

\subsubsection{Proofs of Lemma \ref{lem_cov_mean}, Lemma \ref{lem_cov_var} and Lemma \ref{lem_lindeberg}}
In this section, we give the proofs of the auxiliary results needed in the proof of Theorem \ref{thm_conv_non_feasible}. 
\begin{proof}[Proof of Lemma \ref{lem_cov_mean}]
We decompose
\begin{align*}
    \sum_{k=1}^{n_1 + n_2} \sigma_{n,k}^2  = 
    \sum_{k=1}^{n_1} \E_{k-1} \left[ T_{n,k}^2\right]
    + \sum_{k=n_1+1}^{n_1 + n_2} \E_{k-1} \left[ T_{n,k}^2\right]
\end{align*}
Let $k \leq n_1$ and define
\begin{align} \label{eq_def_B1}
    \bfB_1 = \bfB_{1,k} = \lb \frac{1}{ n_1 - 1} \bfA_{1:(k-1)} + \bfSigma_1 - \bfSigma_{2} \rb. 
\end{align}
Using \eqref{eq_t_nk}, Lemma \ref{lem_quad_form_formula} \ref{lem_part1} and Assumption \ref{ass_model}, we have 
\begin{align}
& \E_{k-1} [ T_{n,k}^2(t) ] \nonumber \\
   &=  \frac{4}{n_1^2  }\E_{k-1} \lb \bfx_k^\top \bfB_1 \bfx_k 
    - \tr \bfSigma_1 \bfB_1 \rb^2 \nonumber \\  
    & = \frac{4}{n_1^2  } \Big\{ \frac{\E [ \xi_{1,1}^4 ]}{p(p+2)} \lb \lb \tr  \bfSigma_1 \bfB_1 \rb ^2  + 2 \tr \lb \bfSigma_1 \bfB_1 \rb^2 \rb - \lb  \tr  \bfSigma_1 \bfB_1 \rb^2 \Big\}. \label{a1}
\end{align}
Note that 
\begin{align*}
    \E \tr \lb \bfSigma_1 \bfB_1 \rb^2 
   &  = \frac{1}{ ( n_1 - 1 ) ^2 } \E \left[ \tr \lb  \bfSigma_1 \bfA_{1:(k-1)} \rb^2 \right] + \tr \left\{ \bfSigma_1 \lb \bfSigma_1 - \bfSigma_2 \rb \right\} ^2, \\ 
    \E \tr^2 \lb \bfSigma_1 \bfB_1 \rb 
   &  = \frac{1}{ ( n_1 - 1 ) ^2 } \E \left[ \tr^2 \lb  \bfSigma_1 \bfA_{1:(k-1)} \rb \right] + \tr^2 \lb \bfSigma_1 \lb \bfSigma_1 - \bfSigma_2 \rb \rb .
\end{align*}
Using Lemma \ref{lem_tr_sq}, Lemma \ref{lem_tr} and Assumption \ref{ass_cov}, we see that the term 
\begin{align} \label{eq_1stcase}
\frac{1}{ n_1^2 ( n_1 - 1 ) ^2 } \lb \frac{\E [ \xi_{1,1}^4 ]}{p(p+2)} - 1\rb   \E \left[ \tr^2 \lb  \bfSigma_1 \bfA_{1:(k-1)} \rb \right]
\end{align}
does asymptotically not contribute to the mean of \eqref{a1}.
Using these results, we get for the mean of \eqref{a1}
\begin{align}
\sum_{k=1}^{n_1} \E \left[ \E_{k-1} [ T_{n,k}^2(t) ] \right] 
&= \E \left[ \frac{8}{n_1^2 (n_1 - 1)^2 } \sum_{k=1}^{n_1} \tr \lb \bfSigma_1 \bfA_{1: ( k-1 ) }\rb^2 \right] + \frac{8}{n_1} \tr  \left\{ \bfSigma_1 \lb \bfSigma_1 - \bfSigma_2 \rb   \right\} ^2 \nonumber \\ 
 & + \frac{4}{n_1} \lb \frac{\E [ \xi_{1,1}^4 ]}{p(p+2)} - 1\rb  \tr^2 \lb \bfSigma_1 \lb \bfSigma_1 - \bfSigma_2 \rb \rb \nonumber \\
 & \quad 
    + o \lb \sum_{k=1}^{n_1} \E \left[ \E_{k-1} [ T_{n,k}^2(t) ] \right] \rb \\ 
& = 
    \frac{4}{n_1^2} \tr^2 \bfSigma_1^2 
    + \frac{8}{n_1} \tr  \left\{ \bfSigma_1 \lb \bfSigma_1 - \bfSigma_2 \rb   \right\} ^2 \nonumber \\ 
    & + \frac{4}{n_1} \lb \frac{\E [ \xi_{1,1}^4 ]}{p(p+2)} - 1\rb  \tr^2 \lb \bfSigma_1 \lb \bfSigma_1 - \bfSigma_2 \rb \rb \nonumber \\ 
    & \quad 
    + o \lb \sum_{k=1}^{n_1} \E \left[ \E_{k-1} [ T_{n,k}^2(t) ] \right] \rb  ,
     \label{e1}
\end{align}
where used Lemma \ref{lem_tr} for the last step. 
We now turn to the case $k\geq n_1+1.$ For this purpose, we define
\begin{align}
    \bfB_{2,k} = \bfB_2 = \lb \frac{1}{ n_2 -1}  \bfA_{(n_1+1):(k-1)} - \frac{1}{n_1} \bfA_{1:n_1} + \bfSigma_2 - \bfSigma_1 \rb.
     = : \mathbf{C}_2 + \bfSigma_2 - \bfSigma_1. \label{eq_def_C}
\end{align}
Using \eqref{eq_t_nk'}, Lemma \ref{lem_quad_form_formula} \ref{lem_part1} and Assumption \ref{ass_model}, we have 
\begin{align}
& \E_{k-1} [ T_{n,k}^2(t) ] \nonumber \\
   &=  \frac{4}{n_2^2  }\E_{k-1} \lb \bfx_k^\top \bfB_2 \bfx_k 
    - \tr \bfSigma_2 \bfB_2 \rb^2 \nonumber \\  
    & = \frac{4}{n_2^2  } \Big\{ \frac{\E [ \xi_{2,1}^4 ]}{p(p+2)} \lb \lb \tr  \bfSigma_2 \bfB_2 \rb ^2  + 2 \tr \lb \bfSigma_2 \bfB_2 \rb^2 \rb - \lb  \tr  \bfSigma_2 \bfB_2 \rb^2 \Big\}. \label{a10}
\end{align}
Now we want to identify the dominating terms in the mean of \eqref{a10}. To this end, we note that 
\begin{align} \label{eq_extra_term}
    \frac{4}{n_2^2  } \lb \frac{\E [ \xi_{2,1}^4 ]}{p(p+2)} - 1 \rb   \E \left[ \tr^2  \bfSigma_2 \bfC_2  \right] 
\end{align}
does asymptotically not contribute to the mean of \eqref{a10},
which follows similarly to the analysis of \eqref{eq_1stcase} in the first case of this proof. 
Using this fact and $\E[\xi_{2,1}^4]=p^2 + \mathcal{O}(p)$, we see that the dominating term in the mean of \eqref{a10} will be 
\begin{align}
    & \frac{8}{n_2^2 } \E  \tr \lb \bfSigma_2 \bfB_2 \rb^2 
    +  \frac{4}{( n- n_1)^2} \lb \frac{\E [ \xi_{2,1}^4 ]}{p(p+2)} - 1\rb  \tr^2 \lb \bfSigma_2 \lb \bfSigma_1 - \bfSigma_2 \rb \rb \nonumber 
    \\ & = \frac{8}{n_2^2 } \E \tr \lb \bfSigma_2 \mathbf{C}_2 \rb^2 
    + \frac{8}{n_2^2 } \E \tr \left\{ \bfSigma_2 \lb \bfSigma_1 - \bfSigma_2 \rb \right\}^2
    + \frac{16}{n_2^2 } \E \tr \lb \bfSigma_2 \mathbf{C}_2 \lb \bfSigma_2 - \bfSigma_1 \rb \rb \nonumber \\
    & \quad  +  \frac{4}{( n- n_1)^2} \lb \frac{\E [ \xi_{2,1}^4 ]}{p(p+2)} - 1\rb  \tr^2 \lb \bfSigma_2 \lb \bfSigma_1 - \bfSigma_2 \rb \rb \\ 
    & =  \frac{8}{n_2^2 } \E \tr \lb \bfSigma_2 \mathbf{C}_2 \rb^2 
    + \frac{8 }{n_2^2 }  \tr \left\{ \bfSigma_2 \lb \bfSigma_1 - \bfSigma_2 \rb \right\}^2 \nonumber 
    \\ & \quad + \frac{4}{( n- n_1)^2} \lb \frac{\E [ \xi_{2,1}^4 ]}{p(p+2)} - 1\rb  \tr^2 \lb \bfSigma_2 \lb \bfSigma_1 - \bfSigma_2 \rb \rb \nonumber \\ 
    &= \frac{8}{ n_2^2 ( n_2 - 1 )^2 } \E \tr \lb \bfSigma_2 \bfA_{(n_1+1):(k-1)} \rb^2 
    +  \frac{8}{ n_2^2 n_1^2 } \E \tr \lb \bfSigma_2 \bfA_{1:n_1} \rb ^2 
    \label{h1}
    \\ & \quad  + \frac{8}{ n_2^2 }  \tr \left\{ \bfSigma_2 \lb \bfSigma_1 - \bfSigma_2 \rb \right\}^2
     +  \frac{4}{( n- n_1)^2} \lb \frac{\E [ \xi_{2,1}^4 ]}{p(p+2)} - 1\rb  \tr^2 \lb \bfSigma_2 \lb \bfSigma_1 - \bfSigma_2 \rb \rb, \nonumber 
\end{align}
where we used 
\begin{align*}
    \E \tr \lb \bfSigma_2 \mathbf{C}_2 \lb \bfSigma_2 - \bfSigma_1 \rb \rb = 0.
\end{align*}
Now note that the two summands in \eqref{h1} are computed in Lemma \ref{lem_tr}. 
This gives
\begin{align}
    & \sum_{k= n_1+1}^{n_1 + n_2} \E \E_{k-1} [T_{n,k}^2] \nonumber \\ 
    & = \frac{4}{n_2^2} \tr^2 \bfSigma_2^2
    + \frac{8}{n_2 n_1} \tr^2 \bfSigma_2 \bfSigma_1 \nonumber \\ & \quad +
    \frac{8}{ n_2 }  \tr \left\{ \bfSigma_2 \lb \bfSigma_1 - \bfSigma_2 \rb \right\}^2
     +  \frac{4}{n- n_1} \lb \frac{\E [ \xi_{2,1}^4 ]}{p(p+2)} - 1\rb  \tr^2 \lb \bfSigma_2 \lb \bfSigma_1 - \bfSigma_2 \rb \rb
    \nonumber \\ & \quad 
    + o \lb \sum_{k= n_1+1}^n \E \E_{k-1} [T_{n,k}^2]  \rb . \label{e2}
\end{align}
The proof concludes by combining \eqref{e1} and \eqref{e2}.
\end{proof}
We continue with the proof of Lemma \ref{lem_lindeberg}. 
\begin{proof}[Proof of Lemma \ref{lem_lindeberg}]
    To begin with, we consider the summands for $k\leq n_1.$ Using the formula for $T_{n,k}$ in \eqref{eq_t_nk} and the definition of $\bfB_{1,k}$ in \eqref{eq_def_B1}, we obtain
    \begin{align}
        \sum_{k=1}^{n_1} \E [ T_{n,k}^4]
        &= \frac{4}{n_1^4} \sum_{k=1}^{n_1} \E \lb \bfx_k^\top \bfB_{1,k} \bfx_k - \tr \bfSigma_1 \bfB_{1,k} \rb^4 \nonumber \\
        & \lesssim \frac{1}{n_1^4} \sum_{k=1}^{n_1} \E \left[  \tr^2 \lb \bfSigma_1 \bfB_{1,k} \rb^2 \right] + \frac{1}{n_1^4} o \lb \frac{1}{p}\rb  \sum_{k=1}^{n_1}  \E \left[ \tr^4 \lb \bfSigma_1 \bfB_{1,k} \rb \right]  , 
        \label{eq_est_t_4thmom}
    \end{align}
    where we used Lemma \ref{lem_quad_form_est} for the inequality. 
    First, we concentrate on bounding the first term, namely,
    \begin{align*}
        S_{n,1} = \frac{1}{n_1^4} \sum_{k=1}^{n_1} \E \left[  \tr^2 \lb \bfSigma_1 \bfB_{1,k} \rb^2 \right] .
    \end{align*}
    By the definition of $\bfB_{1,k}$ in \eqref{eq_def_B1}, we obtain
    \begin{align*}
        S_{n,1} \lesssim S_{n,1,1} + S_{n,1,2},
    \end{align*}
    where
    \begin{align*}
        S_{n,1,1} & = \frac{1}{n_1^4 ( n_1-1 )^4} \sum_{k=1}^{n_1} \E \left[  \tr^2 \lb \bfSigma_1 \bfA_{1:(k-1)} \rb^2 \right] , \\
        S_{n,1,2} & = \frac{1}{n_1^3}   \tr^2 \lb \bfSigma_1 \lb \bfSigma_1 - \bfSigma_2 \rb  \rb^2 .
    \end{align*}
    From \eqref{eq_def_sigma_n}, we see that 
    \begin{align*}
        \sigma_n^4 \gtrsim \frac{1}{n_1^2} \tr^2 \left\{ \bfSigma_1 ( \bfSigma_1 - \bfSigma_2) \right\}^2. 
    \end{align*}
    This implies that $S_{n,1,2} = o ( \sigma_n^4). $ Next,  we turn to $S_{n,1,1}$. Since $\bfx_i \bfx_i^\top = \bfSigma_1$ and $\bfx_i$ and $\bfx_j$ are independent for $i \neq j$, we obtain the representation
    \begin{align*}
         & \E \left[  \tr^2 \lb \bfSigma_1 \bfA_{1:(k-1)} \rb^2 \right] \\ 
        & = \sum_{i,j,l,m=1}^{k-1} \E \left[ \tr \lb \bfSigma_{1} \lb \bfx_i \bfx_i^\top - \bfSigma_1 \rb  \bfSigma_{1} \lb \bfx_j \bfx_j^\top - \bfSigma_1 \rb \rb \tr \lb \bfSigma_{1} \lb \bfx_l \bfx_l^\top - \bfSigma_1 \rb  \bfSigma_{1} \lb \bfx_m \bfx_m^\top - \bfSigma_1 \rb \rb \right] \\ 
        & = \sum_{i,j=1}^{k-1} \E \left[ \tr \lb \bfSigma_{1} \lb \bfx_i \bfx_i^\top - \bfSigma_1 \rb \rb^2 \tr \lb \bfSigma_{1} \lb \bfx_j \bfx_j^\top - \bfSigma_1 \rb \rb^2 \right] 
        \\ & \quad
        + 2 \sum_{i,j=1}^{k-1} \E \left[ \tr^2 \lb \bfSigma_{1} \lb \bfx_i \bfx_i^\top - \bfSigma_1 \rb  \bfSigma_{1} \lb \bfx_j \bfx_j^\top - \bfSigma_1 \rb \rb \right] .
    \end{align*}
   Next, we apply the Cauchy-Schwarz inequality for the Frobenius inner product as well as the fact $(a+b)^2 \lesssim a^2 +b^2$ for $a,b\geq 0$. This gives 
    \begin{align*}
         \E \left[  \tr^2 \lb \bfSigma_1 \bfA_{1:(k-1)} \rb^2 \right] 
        & \lesssim \sum_{i,j=1}^{k-1} \E \left[ \tr \lb \bfSigma_{1} \lb \bfx_i \bfx_i^\top - \bfSigma_1 \rb \rb^2 \tr \lb \bfSigma_{1} \lb \bfx_j \bfx_j^\top - \bfSigma_1 \rb \rb^2 \right] \\
        & =  \E \left[ \sum_{i=1}^{k-1} \tr \lb \bfSigma_{1} \lb \bfx_i \bfx_i^\top - \bfSigma_1 \rb \rb^2  \right]^2
        \\ & = \E \left[ \sum_{i=1}^{k-1} \tr \lb \lb \bfSigma_1 \bfx_i \bfx_i^\top \rb^2 - 2 \bfSigma_1 \bfx_i \bfx_i^\top \bfSigma_1 + \bfSigma_1^4  \rb  \right]^2
        \\ 
        & \lesssim 
        \E \left[ \sum_{i=1}^{k-1} \lb \bfx_i^\top \bfSigma_1 \bfx_i - \tr \bfSigma_1^2 \rb^2\right]^2
         + \E \left[ \sum_{i=1}^{k-1} \lb \bfx_i^\top \bfSigma_1^2 \bfx_i - \tr \bfSigma_1^3 \rb  \right]^2 \\ & \quad 
        + ( k - 1 ) ^2 \tr^4 \bfSigma_1^2 + (k - 1)^2 \tr^2 \bfSigma_1^3 \\
        & \lesssim (k-1)
        \E   \lb \bfx_1^\top \bfSigma_1 \bfx_1 - \tr \bfSigma_1^2 \rb^4
        + (k-1)^2 \left\{ \E   \lb \bfx_1^\top \bfSigma_1 \bfx_1 - \tr \bfSigma_1^2 \rb^2 \right\}^2 \\ & \quad 
         + (k - 1) \E   \lb \bfx_1^\top \bfSigma_1^2 \bfx_1 - \tr \bfSigma_1^3 \rb ^2 \\ & \quad 
        + ( k - 1 ) ^2 \tr^4 \bfSigma_1^2 + (k - 1)^2 \tr^2 \bfSigma_1^3 .
    \end{align*}
    We are now in the position to apply Lemma \ref{lem_quad_form_est}. This gives for the summands in $S_{n,1,1}$ 
 \begin{align*}
      \E \left[  \tr^2 \lb \bfSigma_1 \bfA_{1:(k-1)} \rb^2 \right]    
        & \lesssim (k - 1) \tr^2 \bfSigma_1^4 + (k - 1) o(1/p) \tr^4 \bfSigma_1^2 
        \\ & \quad + ( k - 1)^2 \lb  \tr \bfSigma_1^4 +  o(1/\sqrt{p}) \tr^2 \bfSigma_1^2 \rb^2  \\ 
        & \quad + (k-1) \tr \bfSigma_1^6 + (k-1) o(1/\sqrt{p}) \tr^2 \bfSigma_1^3 \\
        & \quad 
        + ( k - 1 ) ^2 \tr^4 \bfSigma_1^2 + (k - 1)^2 \tr^2 \bfSigma_1^4 \\ 
        & \lesssim ( k - 1 ) ^2 \tr^4 \bfSigma_1^2 + (k - 1)^2 \tr^2 \bfSigma_1^4 + (k -1) \tr^3 \bfSigma_1^2
       ,
 \end{align*}
 where we used the inequality $\tr\bfSigma_1^6 \leq \tr^3 \bfSigma^2$ for the last estimate.
    Assumption \ref{ass_cov} implies that the last term is of order 
    \begin{align*}
            ( k - 1 ) ^2 \tr^4 \bfSigma_1^2 + o \lb ( k - 1 ) ^2 \tr^4 \bfSigma_1^2 \rb .
    \end{align*}
    Thus, we get for $S_{n,1,1}$
    \begin{align*}
        S_{n,1,1} \lesssim \frac{1}{n_1^5} \tr^4 \bfSigma_1^2 + o \lb \frac{1}{n_1^5} \tr^4 \bfSigma_1^2 \rb .
    \end{align*}
    As a consequence of \eqref{eq_sigma_n_lower_bound}, we note that  
    \begin{align} \label{eq_bound_sigma_power4}
        \sigma_n^4 \gtrsim \frac{1}{n_1^4} \tr^4 \bfSigma_1^2 .
    \end{align}
    Thus, we conclude that  
    \begin{align} \label{eq_s_n11_result}
        S_{n,1,1}= o ( \sigma_n^4).
    \end{align}
    We now turn to the investigation of the second summand in \eqref{eq_est_t_4thmom}, that is,
    \begin{align*}
        S_{n,2} = \frac{1}{n_1^4} o \lb \frac{1}{p}\rb  \sum_{k=1}^{n_1}  \E \left[ \tr^4 \lb \bfSigma_1 \bfB_{1,k} \rb \right].
    \end{align*}
   By using the definition of $\mathbf{B}_{1,k}$ in \eqref{eq_def_B1}, we obtain
    \begin{align*}
        S_{n,2} & \lesssim S_{n,2,1} + S_{n,2,2}, 
    \end{align*}
    where
    \begin{align*}
        S_{n,2,1} & = \frac{1}{n_1^4 ( n_1-1)^4} o \lb \frac{1}{p}\rb  \sum_{k=1}^{n_1}  \E \left[ \tr^4 \lb \bfSigma_1 \bfA_{1:(k-1)} \rb \right], \\
        S_{n,2,2} & = \frac{1}{n_1^4} o \lb \frac{1}{p}\rb  \sum_{k=1}^{n_1}   \tr^4 \lb \bfSigma_1 \lb \bfSigma_1 - \bfSigma_2 \rb \rb .
    \end{align*}
    Similarly to $S_{n,1,2}$, it follows that $S_{n,2,2} = o ( \sigma_n^4). $ In the following, we use the basic fact that for centered and independent random variables $y_1, \ldots, y_m$ it holds
    \begin{align*}
        \E \left[ \lb \sum_{i=1}^m y_i \rb^4  \right] = \sum_{i=1}^m \E [y_i^4] + 3 \sum_{i \neq j} \E[y_{i}^2] \E[y_j^2].
    \end{align*}
    Combining this equality with Lemma \ref{lem_quad_form_est}, we get for the summands of $S_{n,2,1}$
    \begin{align*}
        & \E \left[ \tr^4 \lb \bfSigma_1 \bfA_{1:(k-1)} \rb \right] 
         = \E \left[ \sum_{i=1}^{k-1} \lb \bfx_i^\top \bfSigma_1 \bfx_i - \tr \bfSigma_1^2 \rb  \right]^4 \\ 
        & \lesssim  (k -1) \E  \lb \bfx_1^\top \bfSigma_1 \bfx_1 - \tr \bfSigma_1^2 \rb ^4
        + (k-1)^2 \left\{ \E  \lb \bfx_1^\top \bfSigma_1 \bfx_1 - \tr \bfSigma_1^2 \rb ^2 \right\}^2 \\ 
        & \lesssim (k-1) \tr^2 \bfSigma_1^4 
        + (k-1) o(1/p) \tr^4 \bfSigma_1^2 
        + (k-1)^2 \lb  \tr \bfSigma_1^4 +  o(1/\sqrt{p}) \tr^2 \bfSigma_1^2 \rb^2 \\
        & \lesssim (k-1)^2 \tr^2 \bfSigma_1^4 
        + (k-1)^2 o(1/p) \tr^4 \bfSigma_1^2 .
    \end{align*}
    Similarly to $S_{n,1,1}$, we obtain $S_{n,2,1} = o ( \sigma_n^4 ) . $
    The case $k > n_1$ can be treated similarly, and we omit the details for the sake of brevity. 
\end{proof}

\begin{proof}[Proof of Lemma \ref{lem_cov_var}] 
    To begin with, we show that 
    \begin{align} \label{eq_new_goal}
         \Var \lb \frac{8}{n_1^2} \sum_{k=1}^{n_1} \tr \lb \bfSigma_1 \bfB_{1,k} \rb^2 \rb ,
         \Var \lb \frac{8}{n_2^2} \sum_{k=n_1+1}^{n_1 + n_2} \tr \lb \bfSigma_2 \bfB_{2,k} \rb^2 \rb
         = o (\sigma_n^4). 
    \end{align}
    In the following, we are going to investigate the first term, and note that the second one can be handled similarly. 
    It follows from the definition of $\bfB_{1,k}$ in \eqref{eq_def_B1} that
    \begin{align}
        & \Var \lb \frac{1}{n_1^2} \sum_{k=1}^{n_1} \tr \lb \bfSigma_1 \bfB_{1,k} \rb^2 \rb 
       \nonumber \\ &  = \frac{1}{n_1^4}  
         \Var \lb \frac{1}{(n_1 - 1)^2} \sum_{k=1}^{n_1} \tr \lb \bfSigma_1 \bfA_{1:(k-1)} \rb^2 + \frac{2}{n_1 - 1} \sum_{k=1}^{n_1} \tr  \lb \bfSigma_1 \bfA_{1:(k-1)}  \lb \bfSigma_1 - \bfSigma_2\rb \rb   \rb 
         \nonumber \\ 
        & \lesssim  \frac{1}{n_1^8} \Var \lb  \sum_{k=1}^{n_1} \tr \lb \bfSigma_1 \bfA_{1:(k-1)} \rb^2 \rb 
        + \frac{1}{n_1^6 }\Var \lb \sum_{k=1}^{n_1} \tr  \lb \bfSigma_1 \bfA_{1:(k-1)}  \lb \bfSigma_1 - \bfSigma_2\rb \rb \rb =: V_{1,n}+ V_{2,n} ,
        \label{e4}
    \end{align}
    where $V_{1,n}$ and $V_{2,n}$ are defined in the obvious way. Thus, to prove \eqref{eq_new_goal}, it suffices to show that 
    \begin{align}
        \label{eq_goal_vn}
        V_{1,n}, V_{2,n} = o(\sigma_n^4). 
    \end{align}
    First, we analyze $V_{1,n}$ in \eqref{e4} and get 
    \begin{align}
    & \frac{1}{n_1^8} \Var \lb  \sum_{k=1}^{n_1} \tr \lb \bfSigma_1 \bfA_{1:(k-1)} \rb^2 \rb \nonumber \\ 
        & = \frac{1}{n_1^8} \sum_{k=1}^n \Var \lb   \tr \lb \bfSigma_1 \bfA_{1:(k-1)} \rb^2 \rb
        + \frac{2}{n_1^8} \sum_{1 \leq k<j \leq n_1} \cov \lb \tr \lb \bfSigma_1 \bfA_{1:(k-1)} \rb^2, \tr \lb \bfSigma_1 \bfA_{1:(j-1)} \rb^2 \rb \nonumber \\
        & = \frac{2}{n_1^8} \sum_{1 \leq k<j \leq n_1} \cov \lb \tr \lb \bfSigma_1 \bfA_{1:(k-1)} \rb^2, \tr \lb \bfSigma_1 \bfA_{1:(j-1)} \rb^2 \rb + o ( \sigma_n^4), \label{b2}
    \end{align}
    where the estimate follows from \eqref{eq_s_n11_result}.
    Using the definition of $\bfA_{1:(k-1)}$ in \eqref{eq_def_A_1k}, we obtain
    \begin{align}
        &  \sum_{1 \leq k<j \leq n_1} \cov \lb \tr \lb \bfSigma_1 \bfA_{1:(k-1)} \rb^2, \tr \lb \bfSigma_1 \bfA_{1:(j-1)} \rb^2 \rb \nonumber \\
        & = \sum_{1 \leq k<j \leq n_1} \sum_{m,m'=1}^{k-1} \sum_{l, l' =1}^{j-1} \cov \Big( \tr \left\{ \bfSigma_1 \lb \bfx_m \bfx_m^\top - \bfSigma_1 \rb \bfSigma_1 \lb \bfx_{m'} \bfx_{m'}^\top - \bfSigma_1 \rb    \right\}, \nonumber \\ 
        & \quad \quad \quad \quad ~ \quad \quad \quad \quad \quad \quad \quad \tr \left\{ \bfSigma_1 \lb \bfx_l \bfx_l^\top - \bfSigma_1 \rb \bfSigma_1 \lb \bfx_{l'} \bfx_{l'}^\top - \bfSigma_1 \rb    \right\} \Big) \nonumber \\
        & = 2 \sum_{1 \leq k<j \leq n_1} \sum_{m,l=1}^{k-1} \Var \lb \tr \left\{ \bfSigma_1 \lb \bfx_m \bfx_m^\top - \bfSigma_1 \rb \bfSigma_1 \lb \bfx_{l} \bfx_{l}^\top - \bfSigma_1 \rb    \right\} \rb \nonumber \\
        &= 2 \sum_{k=1}^{n_1} (n_1-k) (k-1) \Var \lb \tr \left\{ \bfSigma_1 \lb \bfx_1 \bfx_1^\top - \bfSigma_1 \rb     \right\}^2 \rb 
        \nonumber \\ & + 
       2  \sum_{k=1}^{n_1} (n_1-k) (k-1)^2 \Var \lb \tr \left\{ \bfSigma_1 \lb \bfx_1 \bfx_1^\top - \bfSigma_1 \rb \bfSigma_1 \lb \bfx_{2} \bfx_{2}^\top - \bfSigma_1 \rb    \right\} \rb \nonumber \\
        & = 2 n_1^3 \Var \lb \tr \left\{ \bfSigma_1 \lb \bfx_1 \bfx_1^\top - \bfSigma_1 \rb     \right\}^2 \rb 
         + 2
        n_1^4 \Var \lb \tr \left\{ \bfSigma_1 \lb \bfx_1 \bfx_1^\top - \bfSigma_1 \rb \bfSigma_1 \lb \bfx_{2} \bfx_{2}^\top - \bfSigma_1 \rb    \right\} \rb \nonumber \\
        & \lesssim n_1^3 \Var \lb \left\{ \bfx_1^\top \bfSigma_1 \bfx_1 \right\}^2 \rb 
        + n_1^3 \Var \lb  \bfx_1^\top \bfSigma_1^3 \bfx_1  \rb
        + n_1^4 \Var \lb \left\{ \bfx_1^\top \bfSigma_1 \bfx_2 \right\}^2 \rb 
        + n_1^4 \Var \lb  \bfx_1^\top \bfSigma_1^3 \bfx_1  \rb  \label{b_help}\\ 
        & \lesssim n_1^3 \Var \lb \left\{ \bfx_1^\top \bfSigma_1 \bfx_1 \right\}^2 \rb 
        + n_1^4 \Var \lb \left\{ \bfx_1^\top \bfSigma_1 \bfx_2 \right\}^2 \rb 
        + n_1^4 \Var \lb  \bfx_1^\top \bfSigma_1^3 \bfx_1  \rb,
        \label{b3}
    \end{align}
    where we used for \eqref{b_help} the crude upper bound $\Var (X+Y) \lesssim \Var(X) + \Var(Y)$ for two random variables $X,Y.$
    Thus, by \eqref{b2} and \eqref{b3}, it is left to show that 
    \begin{align}
        n_1 ^{-5} \Var \lb \left\{ \bfx_1^\top \bfSigma_1 \bfx_1 \right\}^2 \rb & = o ( \sigma_n^4), \label{b4}\\ 
        n_1^{-4} \Var \lb \left\{ \bfx_1^\top \bfSigma_1 \bfx_2 \right\}^2 \rb & = o ( \sigma_n^4), \label{b5}\\ 
        n_1^{-4} \Var \lb  \bfx_1^\top \bfSigma_1^3 \bfx_1  \rb & = o ( \sigma_n^4 ). \label{b6}
    \end{align}
    First, we have by Lemma \ref{lem_quad_form_est} and Lemma \ref{lem_quad_form_formula} \ref{lem_part1}
    \begin{align*}
        \Var \lb \left\{ \bfx_1^\top \bfSigma_1 \bfx_1 \right\}^2 \rb 
        & = \Var \lb \left\{ \bfx_1^\top \bfSigma_1 \bfx_1 - \tr\bfSigma_1^2\right\}^2 + \lb  2 \tr\bfSigma_1^2 \rb  \bfx_1^\top \bfSigma_1 \bfx_1 \rb \\
         & \lesssim \E \lb  \bfx_1^\top \bfSigma_1 \bfx_1 - \tr\bfSigma_1^2\rb^4 + \tr^2\bfSigma_1^2 \E \lb \bfx_1^\top \bfSigma_1 \bfx_1  - \tr\bfSigma_1^2 \rb^2 \\
         & \lesssim \tr^2 \bfSigma_1^4 + o(1/p) \tr^4 \bfSigma_1^2
         + \tr^2\bfSigma_1^2 \lb \tr \bfSigma_1^4 +\frac{1}{p} \tr^2 \bfSigma_1^2 \rb \\ 
         & = o \lb \tr ^4 \bfSigma_1^2\rb. 
    \end{align*}
    Using the bound $\sigma_n^4 \gtrsim n_1^{-4} \tr^4 \bfSigma_1^2$ from \eqref{eq_bound_sigma_power4}, the assertion \eqref{b4} follows. 

    For the term in \eqref{b5}, we use Lemma \ref{lem_quad_form_formula} \ref{lem_part5} and obtain
    \begin{align*}
        n_1^{-4} \Var \lb \left\{ \bfx_1^\top \bfSigma_1 \bfx_2 \right\}^2 \rb & \lesssim n_1^{-4} \E \lb  \bfx_1^\top \bfSigma_1 \bfx_2  \rb ^4 \lesssim n_1^{-4} \lb \tr^2 \bfSigma_1^4 + \tr \bfSigma_1^8 \rb \\ 
        & = o \lb n_1^{-4} \tr^4 \bfSigma_1^2 \rb, 
    \end{align*}
    which is \eqref{b5}. For the last term, we obtain from Lemma \ref{lem_quad_form_formula} \ref{lem_part1} and assumption \ref{ass_cov}
    \begin{align*}
        n_1^{-4} \Var \lb  \bfx_1^\top \bfSigma_1^3 \bfx_1  \rb 
        & = n_1^{-4} \E \lb  \bfx_1^\top \bfSigma_1^3 \bfx_1 - \tr\bfSigma_1^4 \rb^2 \\
        & \lesssim n_1^{-4} \lb \tr \bfSigma_1^8 + \frac{1}{p} \tr^2 \bfSigma_1^4 \rb 
        = o \lb n_1^{-4} \tr^4 \bfSigma_1^2 \rb.
    \end{align*}
    Thus, \eqref{b6} holds true, and  $V_{1,n} = o(\sigma_n^4).$
    To complete the proof of \eqref{eq_goal_vn}, we continue with the analysis of the second summand $V_{2,n}$ in \eqref{e4}.
    \begin{align*}
       V_{2,n} & = \frac{1}{n_1^6 }\Var \lb \sum_{k=1}^{n_1} \tr  \lb \bfSigma_1 \bfA_{1:(k-1)}  \lb \bfSigma_1 - \bfSigma_2\rb \rb \rb \\ 
       & = \frac{1}{n_1^6 } \sum_{k=1}^{n_1}  \Var \lb \tr  \lb \bfSigma_1 \bfA_{1:(k-1)}  \lb \bfSigma_1 - \bfSigma_2\rb \rb \rb \\ & 
       + \frac{2}{n_1^6 } \sum_{1 \leq k<j\leq n_1} \cov \lb \tr  \lb \bfSigma_1 \bfA_{1:(k-1)}  \lb \bfSigma_1 - \bfSigma_2\rb \rb , \tr  \lb \bfSigma_1 \bfA_{1:(j-1)}  \lb \bfSigma_1 - \bfSigma_2\rb \rb \rb \\
       & = \frac{1}{n_1^6 } \sum_{k=1}^{n_1} (k-1) \Var \lb \tr  \lb \bfSigma_1 \bfx_1 \bfx_1^\top  \lb \bfSigma_1 - \bfSigma_2\rb \rb \rb \\ & 
       + \frac{2}{n_1^6 }  \sum_{1 \leq k<j\leq n_1} (k-1) \Var \lb \tr  \lb \bfSigma_1 \bfx_1 \bfx_1^\top  \lb \bfSigma_1 - \bfSigma_2\rb \rb \rb \\
       & \lesssim \frac{1}{n_1^3}  \tr \lb \bfSigma_1^2 \lb \bfSigma_1 - \bfSigma_2\rb \rb^2
       + \frac{1}{n_1^3} o (1/p) \tr^2 \lb \bfSigma_1^2 \lb \bfSigma_1 - \bfSigma_2\rb \rb
       \\ 
       & \lesssim \frac{1}{n_1^3} \tr \lb \bfSigma_1^2 (\bfSigma_1 - \bfSigma_2) \rb^2 
       \lesssim \frac{1}{n_1^3} \tr \lb \bfSigma_1^2 \rb  \tr \lb \bfSigma_1 (\bfSigma_1 - \bfSigma_2) \rb^2 ,
    \end{align*}
    where we applied Lemma \ref{lem_quad_form_formula} \ref{lem_part1} for the third to last inequality, the fact $\tr^2 \bfC \leq p \tr \bfC^2 $ for any $\bfC\in \R^{p \times p}$ for the second to last inequality and the crude inequality $\tr \bfA \bfB \leq \tr \bfA \tr \bfB$ for positive definite matrices $\bfA, \bfB \in \R^{p\times p}$ in the last step.
    From \eqref{eq_def_sigma_n}, we have the lower bound
    \begin{align} \label{eq_ref_lower_bound_sigma}
      \sigma_n^4 \gtrsim \frac{1}{n_1^3} \tr^2   \lb \bfSigma_1^2 \rb  \tr \lb \bfSigma_1 (\bfSigma_1 - \bfSigma_2) \rb^2 .
    \end{align}
    Combining \eqref{eq_ref_lower_bound_sigma} and assumption \ref{ass_cov}, it follows that $V_{2,n} = o(\sigma_n^4)$, and the proof of \eqref{eq_goal_vn} concludes. This completes the proof of \eqref{eq_new_goal}.

Combining \eqref{eq_new_goal} with \eqref{a1} and \eqref{a10}, we find that
\begin{align} \label{q1}
      \E \left[  \sum_{k=1}^{n_1 + n_2}  \sigma_{n,k}^2 \right] - \sum_{k=1}^{n_1 + n_2}  \sigma_{n,k}^2
      = \operatorname{Rem} + o_{\PR}(\sigma_n^2),
\end{align}
where the remainder satisfies 
\begin{align}
   | \operatorname{Rem} | & \lesssim 
  \E \left[ \sum_{k=1}^{n_1} \frac{1}{n_1^4 p} \tr^2 \lb \bfSigma_1 \bfA_{1:(k-1)} \rb   \right]
   + \sum_{k=1}^{n_1} \frac{1}{n_1^4 p} \tr^2 \lb \bfSigma_1 \bfA_{1:(k-1)} \rb  \label{eq_first_line} \\ & 
   + \E \left[ \sum_{k=n_1 + 1}^{n_2} \frac{1}{n_2^2 p} \tr^2 \lb  \bfSigma_2 \bfC_{2,k} \rb  \right]
    +  \sum_{k=n_1 + 1}^{n_2} \frac{1}{n_2^2 p} \tr^2 \lb  \bfSigma_2 \bfC_{2,k} \rb ,\label{eq_second_line}
\end{align}
and we recall the definition of $\bfC_{2,k}= \bfC_2$ in \eqref{eq_def_C}.
For the terms in \eqref{eq_first_line}, we obtain the order $o_{\PR}(\sigma_n^2)$ from Lemma \ref{lem_tr_sq}. The terms in \eqref{eq_second_line} can be treated similarly, which implies that the right hand side of \eqref{q1} is of order $o_{\PR}(\sigma_n^2)$. 
\end{proof}

\subsubsection{Proof of \eqref{eq_remainder_negl}} \label{sec_proof_remainder}
 \begin{proof}[Proof of \eqref{eq_remainder_negl}]
    
We begin with the decomposition 

\begin{align*}
    T_n - \tilde{T}_n 
    & =  \sum_{i=1,2} \lb U_{n,i,2} + U_{n,i,3} \rb
    + V_{n,2} + V_{n,3}+ V_{n,4},
\end{align*}
where $U_{n,i,k}$ and $V_{n,k}$ stand for the $k$th summand appearing the definition of $U_{n,i}$ in \eqref{eq_def_U} and $V_n$ in \eqref{eq_def_V}. More precisely, we define for $i\in \{1,2\}$
\begin{align*}
     U_{n,i,2} & := -\frac{2}{(n_i)_3}\sum_{1 \leq j,k,l \leq n_i  }^{\star} \bfx_j^{(i)^\top}\bfx_k^{(i)}\bfx_k^{(i)^\top}\bfx_{l}^{(i)}   \quad  \\ 
       U_{n,i,3} & := \frac{1}{(n_i)_4}\sum_{1 \leq j, k , l, m \leq n_i }^{\star} \bfx_j^{(i)^\top}\bfx_k^{(i)}\bfx_{l}^{(i)^\top}\bfx_m^{(i)}, \\
       V_{n,2} &: =  - \frac{1}{n_1 n_2 (n_1 -1)} \sum_{1 \leq k,l \leq n_1}^\star \sum_{j=1}^{n_2} \bfx_l^{(1)^\top} \bfx_j^{(2)} \bfx_j^{(2)^\top} \bfx_k^{(1)}, \\
    V_{n,3} := &  -   \frac{1}{n_2 n_1 (n_2 -1)} \sum_{1 \leq k,l \leq n_2}^\star \sum_{j=1}^{n_1} \bfx_l^{(2)^\top} \bfx_j^{(1)} \bfx_j^{(1)^\top} \bfx_k^{(2)} , \\ 
    V_{n,4} & := \frac{1}{(n_1)_2 (n_2)_2} \sum_{1 \leq i,k \leq n_1}^\star \sum_{1 \leq j,l \leq n_2}^\star 
    \bfx_i^{(1)^\top} \bfx_j^{(2)} \bfx_k^{(1)^\top} \bfx_l^{(2)^\top} .
\end{align*}
Recalling that $\boldsymbol{\mu}_{n,1} =\boldsymbol{\mu}_{n,2} = \mathbf{0}, $ we have 
\begin{align*}
    \E [ U_{n,i,2}] = \E [ U_{n,i,3}] = \E [ V_{n,k}] = 0 , \quad i\in \{1,2\}, ~ k \in \{2,3,4\}. 
\end{align*}
Thus, it suffices to show that 
\begin{align}
    \Var ( U_{n,i,2}) , \Var (V_{2,n}), \Var(V_{n,3}) & = o (\sigma_n^2), \label{goal1}\\
    \Var ( U_{n,i,3}) , \Var (V_{n,4}) = o (\sigma_n^2). \label{goal2}
\end{align} 
For symmetry reasons, we may concentrate on the case $i=1$ and, to facilitate notation, we again write $\bfx_k^{(1)}=\bfx_k$ for $1 \leq k \leq n_1$. 
    Combining the identities in Lemma \ref{lem_quad_form_formula} \ref{lem_part2}-\ref{lem_part3} with assumption \ref{ass_model}, we obtain 
    \begin{align} 
        \Var ( U_{n,1, 2} ) 
       &  \lesssim \frac{1}{n_1^6} \sum_{1 \leq j,k,l \leq n_1  }^{\star}
        \sum_{1 \leq j',k',l' \leq n_1  }^{\star}
        \cov \lb  \bfx_j^\top \bfx_k \bfx_k^\top \bfx_l ,\bfx_{j'}^\top \bfx_{k'} \bfx_{k'}^\top \bfx_{l'} \rb 
       \label{eq_cov} \\& 
        \lesssim \frac{1}{n^2} \Var  \lb \bfx_1^{\top} \bfx_2 \bfx_2^\top \bfx_3 \rb
        = 
        \frac{1}{n_1^2} \left\{ \frac{\E \xi_{1,1}^4}{p(p+2)} \lb \tr^2 \bfSigma_{1}^2 + 2 \tr \bfSigma_{1}^4 \rb - \tr^2\bfSigma_1^2 \right\}  \nonumber \\
        & \lesssim   \frac{1}{n_1^2 p} \tr^2 \bfSigma_{n,1}^2 +  \frac{1}{n_1^2}\tr\bfSigma_1^4
        = o \lb \frac{1}{n_1^2} \tr^2 \bfSigma_{1}^2 \rb = o(\sigma_n^2).
        \label{eq_bound_un112}
    \end{align}
    Here, we note that covariance terms in \eqref{eq_cov} are zero if $j,l\notin \{j',l',k'\}$ or $j',l' \notin \{ j,l,k\}$, as the random vectors $\bfx_1,\ldots, \bfx_{n_1}$ are independent and centered. 
    For the last estimate in \eqref{eq_bound_un112}, we used the fact that $\sigma_n^2 \gtrsim \frac{1}{n_1^2} \tr^2 \bfSigma_{n,1}^2.$ We note that the terms $V_{n,2}$ and $V_{n,3}$ can be handled similarly as they also include a summation over three pairwise indices. Thus, the proof of \eqref{goal1} conludes, and it suffices to show \eqref{goal2} (the case of four pairwise indices). 
    For similar reasons as above, we restrict ourselves to the proof of  $\Var ( U_{n,1,3}) = o(\sigma_n^2)$ and note the analogue assertion for $V_{n,4}$ can be shown similarly.  
    Combining Lemma \ref{lem_quad_form_formula} \ref{lem_part4} with assumption \ref{ass_model}, we obtain
    \begin{align} \label{eq_cov2}
        \Var (U_{n,1,3}) & \lesssim 
        \frac{1}{n_1^8}\sum_{1 \leq j, k , l, m \leq n_i }^{\star} \sum_{1 \leq j', k' , l', m' \leq n_i }^{\star} 
    \cov \lb \bfx_j \bfx_k \bfx_{l} \bfx_m , \bfx_{j'} \bfx_{k'} \bfx_{l'} \bfx_{m'} \rb 
        \\
        & \lesssim  \frac{1}{n_1^4}   \E \lb \bfx_1^\top \bfx_2 \bfx_3^\top \bfx_4 \rb^2 
        = \frac{1}{n_1^4} \tr^2 \bfSigma_{n,1}^2 = o (\sigma_{n}^2),  \label{eq_bound_un113}
    \end{align}
    Here, we note that covariance terms in \eqref{eq_cov2} are zero if $|\{i,j,k,l,i',j',k',l\}|>4$, since the random vectors $\bfx_1, \ldots \bfx_{n_1}$ are independent and centered. 
    Also, we used $\sigma_n^2 \gtrsim \frac{1}{n_1^2} \tr^2 \bfSigma_{n,1}^2$ for the last estimate in \eqref{eq_bound_un113}.
    This concludes the proof of \eqref{eq_remainder_negl}.
 \end{proof}
 
\subsection{Proofs of Proposition \ref{prop_consistent_sigma} and Theorem \ref{thm_power}} \label{sec_proof_power_consistent}

\begin{proof}[Proof of Proposition \ref{prop_consistent_sigma}]
It is sufficient to show the first assertion of the theorem, as the second one is an immediate consequence. 
To this end, we assume $i=1$ for convenience.   We may again assume w.l.o.g. that $\boldsymbol{\mu}_i = \mathbf{0.}$ 
It is straightforward to show that $\frac{2}{n_1}U_{n,1}$ is an unbiased estimator for $\frac{2}{n_1} \tr \bfSigma_{n,i}^2$ for $i \in \{1,2\}$. Recall the decomposition $U_{n,1} = U_{n,1,1} + U_{n,1,2} + U_{n,1,3} $, where $U_{n,1,1}$ is defined in \eqref{def_u_ni1}.  From \eqref{eq_bound_un112} and \eqref{eq_bound_un113}, we know that 
\begin{align*}
    \Var (U_{n,1,2} ), \Var (U_{n,1,3} ) = o \lb  \tr^2 \bfSigma_1^2 \rb. 
\end{align*}
For the variance of $U_{n,1,1}$, we have
\begin{align}
    \Var ( U_{n,1,1} ) & \lesssim \frac{1}{n_1^2} \Var \lb \lb \bfx_1^\top \bfx_2\rb^2 \rb + \frac{1}{n_1} \cov \lb \lb \bfx_1^\top \bfx_2\rb^2, \lb \bfx_1^\top \bfx_3\rb^2\rb 
    \lesssim \frac{1}{n_1} \Var \lb \lb \bfx_1^\top \bfx_2\rb^2 \rb  \label{t1}
    .
\end{align}
We also note that, under assumption \ref{ass_model},
 $$ \E \xi_{1,1}^4 = p^2 + \tau_i p+ o(p).$$
For upper bound in \eqref{t1}, we obtain, using Lemma \ref{lem_quad_form_formula} \ref{lem_part5} and assumption \ref{ass_model},
\begin{align*}
   \frac{1}{n_1}  \Var \lb \lb \bfx_1^\top \bfx_2\rb^2 \rb 
    & \lesssim \frac{1}{n_1} \E \lb  \bfx_1^\top \bfx_2\rb^4
    = \frac{3}{n_1}  \lb \frac{\E \xi_{1,1}^4}{p(p+2)} \rb^2 \lb \tr^2 \bfSigma_1^2 + \frac{2}{n_1^2} \tr \bfSigma_1^4\rb 
     = \mathcal{O} \lb \frac{1}{n_1 } \tr^2 \bfSigma_1^2 \rb  \\ &  = o \lb  \tr^2 \bfSigma_1^2 \rb .
\end{align*}
This concludes the proof. 
\end{proof}

\begin{proof}[Proof of Theorem \ref{thm_power}]
 As a preparatory step, we show that 
\begin{align} \label{eq_bound_r}
     \frac{\tilde \sigma_n}{\sigma_n} \lesssim 1.
\end{align}
To this end, we will first verify that the variance $\sigma_n^2$ defined in \eqref{eq_def_sigma_n} satisfies
     \begin{align}
         \sigma_n^2 \geq \frac{4}{n_1^2} \tr^2 \lb \bfSigma_{n,1}^2\rb + \frac{4}{n_2^2} \tr^2 \lb \bfSigma_{n,2}^2\rb. \label{eq_sigma_n_lower_bound}
    \end{align}
 To see this, we first note that $\tr ^2 \bfA \leq p \tr \bfA^2$ for any $p \times p$ matrix $\bfA$ as a consequence from the Cauchy-Schwarz inequality. Using also that $\tau_i\geq 0,$ we obtain  
     \begin{align}
        &    \frac{8}{n_i} \tr  \left\{ \bfSigma_{n,i} \lb \bfSigma_{n,1} - \bfSigma_{n,2} \rb   \right\} ^2   
        + \frac{4 (\tau_i - 2) }{p n_i}  \tr^2 \lb \bfSigma_{n,i} \lb \bfSigma_{n,1} - \bfSigma_{n,2} \rb \rb 
      \nonumber   \\ & \geq \frac{8}{n_i} \tr  \left\{ \bfSigma_{n,i} \lb \bfSigma_{n,1} - \bfSigma_{n,2} \rb   \right\} ^2   
        - \frac{8  }{p n_i}  \tr^2 \lb \bfSigma_{n,i} \lb \bfSigma_{n,1} - \bfSigma_{n,2} \rb \rb 
        \geq 0. \label{eq_ineq_sigma_term_positive}
     \end{align}
     Then, \eqref{eq_sigma_n_lower_bound} follows from the definition of $\sigma_n^2$ in \eqref{eq_def_sigma_n}. 
   
    Let $k_n = n_1 / (n_1 + n_2). $ Using \eqref{eq_sigma_n_lower_bound}, it is not hard to show that (see p. 914 in \cite{lichen2012})
    \begin{align*}
        \frac{ \tilde\sigma_n }{\sigma_n} \leq \frac{1}{k_n ( 1-k_n)} \lesssim 1. 
    \end{align*}
    This concludes the proof of \eqref{eq_bound_r}.

    Now define $\mu_n = \| \bfSigma_{n,1} - \bfSigma_{n,2} \|_F^2.$  We decompose
    \begin{align*}
            \frac{1}{\hat\sigma_n}T_n
           & = \frac{1}{\hat\sigma_n} \lb T_n - \mu_n \rb + \frac{\mu_n}{\hat\sigma_n}
            = \frac{\tilde\sigma_n}{\hat\sigma_n} \frac{\sigma_n}{\tilde\sigma_n} \frac{1}{\sigma_n} (T_n - \mu_n) + \frac{\mu_n}{\tilde\sigma_n} \frac{\tilde\sigma_n}{\hat\sigma_n}
            = \frac{\mu_n}{\tilde\sigma_n} \frac{\tilde\sigma_n}{\hat\sigma_n} \lb \frac{\sigma_n}{\mu_n} \frac{1}{\sigma_n} (T_n - \mu_n) +1  \rb \\
            & = \frac{\mu_n}{\tilde\sigma_n}  \lb \frac{\sigma_n}{\mu_n} \mathcal{O}_{\PR}(1)  +1  \rb \mathcal{O}_{\PR}(1),
    \end{align*}
    where we used Proposition \ref{prop_consistent_sigma} and Theorem \ref{thm_conv_non_feasible} in the last step. 
    Since $\mu_n / \tilde \sigma_n \gtrsim \mu_n / \sigma_n \to \infty$ by \eqref{eq_bound_r} and our assumption, we have
    \begin{align*}
       \lim_{n\to\infty} \PR \lb\frac{1}{\hat\sigma_n}T_n > z_{1-\alpha} \rb = 1.  
    \end{align*}
\end{proof}

  \section{Background results} \label{sec_appendix_b}
We first collect formulas for mixed moments of bilinear forms in elliptical random vectors. 
\begin{lemma} \label{lem_quad_form_formula}
 Suppose that the $p-$dimensional random vectors $\bfx, \bfx_1, \ldots, \bfx_4$ are independent and satisfy \eqref{eq_em} with $\boldsymbol{\mu}=\mathbf{0}.$
 \begin{enumerate}[label=(\alph*)]
     \item \label{lem_part1}
     Let $\bfC_1, \bfC_2$ be deterministic symmetric $p\times p$ matrices. Then, we have
    \begin{align*}
        & \E \lb \bfx^\top  \bfC_1 \bfx - \tr \bfSigma \bfC_1 \rb \lb \bfx^\top  \bfC_2 \bfx - \tr \bfSigma \bfC_2 \rb \\ 
        & = \frac{\E [ \xi^4 ]}{p(p+2)} \lb \tr \bfSigma \bfC_1 \tr \bfSigma\bfC_2 
        +2 \tr \bfSigma \bfC_1 \bfSigma \bfC_2 \rb - \tr \bfSigma \bfC_1 \tr \bfSigma \bfC_2.
    \end{align*}
    \item \label{lem_part2} It holds
\begin{align*}
    \E \left[ \lb \bfx_1^{\top} \bfx_2 \bfx_2^\top \bfx_3 \rb^2 \right] 
    = \frac{\E \xi_{}^4}{p(p+2)} \lb \tr^2 \bfSigma^2 + 2 \tr \bfSigma^4 \rb. 
    \end{align*}
    \item \label{lem_part3} 
    It holds 
    \begin{align*}
        \E \left[ \lb \bfx_1^\top \bfx_2\rb^2\right] = \tr \bfSigma^2.
    \end{align*}
    \item \label{lem_part4} 
    It holds
    \begin{align*}
        \E \left[ \lb \bfx_1^\top \bfx_2 \bfx_3^\top \bfx_4 \rb^2 \right]  = \tr^2 \bfSigma^2. 
    \end{align*}
    \item \label{lem_part5} 
    It holds
    \begin{align*}
        \E \left[ \lb \bfx_1^\top \bfx_2 \rb ^4 \right] = 3 \lb \frac{\E [\xi^4]}{p(p+2)} \rb^2 \left\{  \tr^2 \lb \bfSigma^2 \rb + 2 \tr \lb \bfSigma^4\rb \right\}  .
    \end{align*}
 \end{enumerate}

\end{lemma}
Part \ref{lem_part1} is an immediate consequence of \cite[Lemma A.1]{hu_et_al_2019}, Part 
\ref{lem_part3} can be verified by a direct computation. Parts \ref{lem_part2}, \ref{lem_part4}, \ref{lem_part5} are given in \cite[Lemma D.2]{wang2023bootstrap}. 

The following lemma provides estimates for higher-order moment of centered quadratic forms in elliptical random vectors. 
\begin{lemma}[Lemma D.4 in \cite{wang2024testing}]
\label{lem_quad_form_est}
    Let $\bfC$ be a deterministic symmetric $p\times p$ matrix. Suppose that the $p-$dimensional random vector $\bfx$ satisfies \eqref{eq_em} with $\boldsymbol{\mu}=\mathbf{0}$ and $\xi$ satisfying condition \eqref{eq_ass_xi}. 
    Then, it holds for any $1 \leq q \leq 8,$
    \begin{align*}
        \E \lb \bfx^\top  \bfC \bfx - \tr \bfSigma \bfC  \rb^q
        \lesssim \tr^{q/2} \lb \bfC \bfSigma \rb^2 
        + o \lb \frac{1}{p^{q/4}} \rb \tr^q \lb  \bfC \bfSigma \rb.
    \end{align*}
\end{lemma}

Next, we compute terms arising in the asymptotic variance of $T_n$ in the proof of Lemma \ref{lem_cov_mean}. To this end, recall the definition of $\bfA_{1:(k-1)}$ in \eqref{eq_def_A_1k}. 
\begin{lemma} \label{lem_tr_sq}
    Under the assumptions of Theorem \ref{thm_conv_non_feasible}, it holds
     $$ \E \left[ \frac{1}{n_1^4 } \sum_{k=1}^{n_1 } \lb \tr  \bfSigma_1 \bfA_{1: (k-1) }\rb ^2 \right]
     =  
     o \lb \frac{1}{n_1^2} \tr^2 \bfSigma_1^2 \rb. $$
\end{lemma}
\begin{proof}[Proof of Lemma \ref{lem_tr_sq}]
    Using Lemma \ref{lem_quad_form_formula} \ref{lem_part1}, we obtain 
    \begin{align*}
         \E \left[  \lb \tr  \bfSigma_1 \bfA_{1: (k-1) }\rb ^2 \right] 
         &= \sum_{i=1}^{k-1} \E \lb \bfx_i^\top \bfSigma_1 \bfx_i - \tr \bfSigma_1^2 \rb ^2 \\
         &=  \sum_{i=1}^{k-1} \left\{ \frac{\E \xi_{1,1}^4}{p ( p +2)} \lb 2 \tr \bfSigma_1^4 +  \tr^2 \bfSigma_1^2 \rb  - \tr^2 \bfSigma_1^2 \right\} \\ 
         &= (k-1) \left\{ \frac{\E \xi_{1,1}^4}{p ( p +2)} \lb 2 \tr \bfSigma_1^4 +  \tr^2 \bfSigma_1^2 \rb  - \tr^2 \bfSigma_1^2 \right\}.
    \end{align*}
    Using $\sum_{k=1}^{n_1} (k-1) = (n_1-1) n_1 /2$, \ref{ass_model} and \ref{ass_cov}, this implies 
    \begin{align*}
        & \E \left[ \frac{1}{n_1^4 } \sum_{k=1}^{n_1 } \lb \tr  \bfSigma_1 \bfA_{1: (k-1) }\rb ^2 \right] \\ 
        & = \frac{1}{n_1^4 } \sum_{k=1}^{n_1 }(k-1) \left\{ \frac{\E \xi_{1,1}^4}{p ( p +2)} \lb 2 \tr \bfSigma_1^4 +  \tr^2 \bfSigma_1^2 \rb  - \tr^2 \bfSigma_1^2 \right\} \\
         & = \frac{(n_1 -1)}{2 n_1^3 }  \left\{ \frac{\E \xi_{1,1}^4}{p ( p +2)} \lb 2 \tr \bfSigma_1^4 +  \tr^2 \bfSigma_1^2 \rb  - \tr^2 \bfSigma_1^2 \right\} \\ 
         &  = \frac{(n_1 -1)}{2 n_1^3 }  \left\{
            (1+o(1)) \tr\bfSigma_1^4 + 
            \frac{\tau_1-2 +o(1)}{p+2} \tr^2 \bfSigma_1^2
         \right\}  \\ 
         & =   o \lb \frac{1}{n^2} \tr^2 \bfSigma_1^2 \rb.
    \end{align*}
\end{proof}

\begin{lemma} \label{lem_tr}
 Under the assumptions of Theorem \ref{thm_conv_non_feasible}, it holds for 
      $k \leq n_1,$ 
     $$  \E \left[ \frac{4}{n_1^2 (n_1 - 1)^2 } \sum_{k=1}^{n_1} \tr \lb \bfSigma_1 \bfA_{1: ( k-1 ) }\rb^2 \right] = \frac{2}{n_1^2} \tr^2 \bfSigma_1^2 + o \lb \frac{1}{n_1^2} \tr^2 \bfSigma_1^2 \rb,  $$
     and for $k> n_1$,
      \begin{align*}  
      & \E \left[ \frac{4}{n_2^2 (n_2 - 1)^2 } \sum_{k=n_1+1}^{n} \tr \lb \bfSigma_1 \bfA_{(n_1+1): ( k-1 ) }\rb^2 \right] \\ 
      & = \frac{2}{n_2^2} \tr^2 \bfSigma_2^2 + o \lb \frac{1}{n_2^2} \tr^2 \bfSigma_2^2 \rb .
      \end{align*}
      Moreover, we have
      \begin{align*}
 \E \left[ \frac{8}{ n_2^2 n_1^2 } \tr \lb \bfSigma_2 \bfA_{1:n_1} \rb ^2 \right] = \frac{8}{ n_2^2 n_1 } \tr^2 \bfSigma_1 \bfSigma_2 + o \lb \frac{1}{ n_2^2 n_1 } \tr^2 \bfSigma_1 \bfSigma_2\rb .  
      \end{align*}
\end{lemma}
\begin{proof}[Proof of Lemma \ref{lem_tr}]
We first use Lemma \ref{lem_quad_form_formula} \ref{lem_part1} and get
    \begin{align*}
         \E \left[ \tr \lb \bfSigma_1 \bfA_{1: ( k-1 ) }\rb^2  \right] 
          &= \sum_{i=1}^{k-1} \E \tr \lb \bfSigma_1 \lb \bfx_i \bfx_i^\top - \bfSigma_1 \rb \rb^2 \\
          &= \sum_{i=1}^{k-1} \E \left[ \lb \bfx_i^\top \bfSigma_1 \bfx_i \rb ^2 
          - 2 \bfx_i^\top \bfSigma_1^3 \bfx_i + \tr \bfSigma_1^4
          \right] \\
          &= \sum_{i=1}^{k-1} \E \left[ \lb \bfx_i^\top \bfSigma_1 \bfx_i \rb ^2 
           - \tr \bfSigma_1^4
          \right] \\ 
         & =  \sum_{i=1}^{k-1} \E \left[ \lb \bfx_i^\top \bfSigma_1  \bfx_i - \tr \bfSigma_1^2 \rb ^2 
          - \tr^2 \bfSigma_1^2 
          + 2 \lb \tr \bfSigma_1^2 \rb \bfx_i^\top \bfSigma_1 \bfx_i 
           - \tr \bfSigma_1^4
          \right] \\
            & =  \sum_{i=1}^{k-1} \E \left[ \lb \bfx_i^\top \bfSigma_1  \bfx_i - \tr \bfSigma_1^2 \rb ^2 
          + \tr^2 \bfSigma_1^2 
           - \tr \bfSigma_1^4
          \right] \\
          &= (k -1 ) \Big\{ \frac{\E \xi_{1,1}^4}{p(p+2)} 
          \left[ \tr^2 \bfSigma_1^2 + 2 \tr \bfSigma_1^4  \right] -  \tr \bfSigma_1^4 
          \Big\}.
    \end{align*}
    Using $\sum_{k=1}^{n_1} (k-1)= n_1 (n_1-1)/2$, \ref{ass_model} and \ref{ass_cov}, this implies
    \begin{align*}
        & \E \left[ \frac{4}{n_1^2 (n_1 - 1)^2 } \sum_{k=1}^{n_1} \tr \lb \bfSigma_1 \bfA_{1: ( k-1 ) }\rb^2 \right] \\ 
        & = \frac{2}{n_1 ( n_1 -1)} \Big\{ \frac{\E \xi_{1,1}^4}{p(p+2)} 
          \left[ \tr^2 \bfSigma_1^2 + 2 \tr \bfSigma_1^4  \right] -  \tr \bfSigma_1^4 
          \Big\} \\
          & =  \frac{2}{n_1^2} \tr^2 \bfSigma_1^2 + o \lb \frac{1}{n^2}  \tr^2 \bfSigma_1^2 \rb  .
    \end{align*}
      This shows the first assertion of the lemma. The second one follows for symmetry reasons. The third assertion follows very similar to the first one, and we do not repeat the arguments here for the sake of brevity. 
\end{proof}

\end{appendix}


\begin{funding}
The author was supported by the Aarhus University Research Foundation (AUFF), project numbers 47221 and 47388.
\end{funding}

\bibliographystyle{imsart-nameyear} 
\bibliography{references.bib}       


\end{document}